\RequirePackage[l2tabu, orthodox]{nag}
\documentclass[a4paper,11pt]{amsart}
\usepackage[margin=32mm,bottom=40mm]{geometry}

\usepackage{lmodern}
\usepackage[stretch=10]{microtype}
\usepackage[T1]{fontenc}
\usepackage{tikz, tikz-cd, stmaryrd, amsmath, amsthm, amssymb,
hyperref, bbm, mathtools, mathrsfs}
\usepackage[shortlabels]{enumitem}
\usetikzlibrary{arrows}
\usetikzlibrary{positioning}
\usepackage[utf8x]{inputenc}
\newcommand{\bb}{\textbf}

\newcommand{\ol}{\overline}
\newcommand{\mc}{\mathcal}
\newcommand{\wh}{\widehat}

\newcommand{\mf}{\mathfrak}
\newcommand{\ms}{\mathscr}

\newcommand{\HH}{\mathbb{H}}
\newcommand{\ZZ}{\mathbb{Z}}

\newcommand{\PP}{\mathbb{P}}

\newcommand{\FF}{\mathbb{F}}
\newcommand{\VV}{\mathbb{V}}

\DeclareMathOperator{\Ind}{Ind}

\DeclareMathOperator{\tr}{tr}

\DeclareMathOperator{\Divv}{Div}

\DeclareMathOperator{\coker}{coker}

\DeclareMathOperator{\id}{id}

\DeclareMathOperator{\Span}{Span}
\DeclareMathOperator{\res}{res}
\DeclareMathOperator{\Gl}{Gl}

\DeclareMathOperator{\ord}{ord}

\DeclareMathOperator{\Gal}{Gal}

\DeclareMathOperator{\cha}{char}

\DeclareMathOperator{\pole}{pole}
\theoremstyle{plain}
\newtheorem{Theorem}{Theorem}[section]
\newtheorem{Lemma}[Theorem]{Lemma}
\newtheorem{Corollary}[Theorem]{Corollary}

\newtheorem{Proposition}[Theorem]{Proposition}

\newtheorem{Example}[Theorem]{Example}

\newtheorem{Question}[Theorem]{Question}

\theoremstyle{definition}

\theoremstyle{remark}

\numberwithin{equation}{section}
\begin{document}

\title[$p$-group Galois covers... II]{$p$-group Galois covers of curves\\ in characteristic~$p$~II}
\author[J. Garnek]{J\k{e}drzej Garnek}
\address{\parbox{\linewidth}{Institute of Mathematics of Polish Academy of Sciences,\\ ul. \'{S}niadeckich 8, 00-656 Warszawa \\ \mbox{}}}
\email{jgarnek@amu.edu.pl}
\urladdr{http://jgarnek.faculty.wmi.amu.edu.pl/}
\date{}  

\begin{abstract}
	Let $k$ be an algebraically closed field of characteristic $p > 0$ and let $G$ be a finite $p$-group. The results of Harbater, Katz and Gabber associate to every $k$-linear action of $G$ on $k[[t]]$ an HKG-cover, i.e. a $G$-cover of the projective line ramified only over $\infty$. In this paper we relate the HKG-covers to the classical problem of determining the equivariant structure of cohomologies of a curve with an action of $G$. To this end, we 
	present a new way of computing cohomologies of HKG-covers. As an application of our results, we compute the equivariant structure of the de~Rham cohomology of Klein four covers in characteristic~$2$.
\end{abstract}

\maketitle
\bibliographystyle{plain}

\section{Introduction} \label{sec:intro}
Studying the equivariant structure of the cohomologies of a curve $X$ over a field~$k$ with an action of a finite group $G$ is a natural and well-researched topic. In  the  classical  case,  that  is, when $\cha k \nmid \# G$, the equivariant structure of the module of holomorphic differentials was completely determined  by  Chevalley and Weil using the character theory of finite groups, cf. \cite{Chevalley_Weil_Uber_verhalten}.
When $\cha k = p > 0$ and $p| \# G$, the structure of $H^0(X, \Omega_X)$ becomes much more complicated. Recall that in this setting the character theory is of limited use and classifying indecomposable representations is a ``wild'' problem. Even indecomposable representations of $G = \ZZ/p \times \ZZ/p$ in characteristic $p > 2$ are thought to be impossible to classify (cf. \cite{Carlson_Friedlander_Suslin_Modules_for_Zp_Zp}). In the tame or weakly ramification case, one may
obtain some information on the image of $H^0(X, \Omega_X)$ in the K-theory of $G$, cf. e.g. \cite{FW_Kock_Galois_module_theory}. Moreover, there are several results
for specific groups (see e.g. \cite{Valentini_Madan_Automorphisms}, \cite{WardMarques_HoloDiffs}, \cite{Bleher_Camacho_Holomorphic_differentials}) or curves (cf. \cite{Lusztig_Coxeter_orbits}, \cite{Dummigan_99},
\cite{Gross_Rigid_local_systems_Gm}, \cite{laurent_kock_drinfeld}). Most of those results are proven by giving an explicit basis
of the cohomology of $X$. In this article and in its prequel \cite{Garnek_p_gp_covers} we propose a different strategy for investigating the equivariant structure of cohomologies in the case when $G$ is a $p$-group.\\

From now on, we assume that $k$ is an algebraically closed field of characteristic~$p$ and~$G$ is a finite $p$-group. Let $X$ be a smooth projective curve over $k$ with an action of $G$. As explained in~\cite{Garnek_p_gp_covers}, we predict that the Hodge cohomology $H^1_{Hdg}(X) := H^0(X, \Omega_X) \oplus H^1(X, \mc O_X)$ and the de~Rham cohomology $H^1_{dR}(X)$
should decompose as $k[G]$-modules into certain global and local parts:
\begin{align*}
	H^1_{Hdg}(X) &\cong \textrm{(global part)} \oplus \bigoplus_{Q \in Y(k)} H^1_{Hdg, Q},\\
	H^1_{dR}(X) &\cong \textrm{(global part)} \oplus \bigoplus_{Q \in Y(k)} H^1_{dR, Q},
\end{align*}
where $\pi : X \to X/G =: Y$ is the quotient morphism. More precisely, the global part should depend only on the ``topology'' of the cover $\pi$ (i.e. on the curve $Y$ and the stabilizer subgroups) and be the same for both cohomologies.
Moreover, for any given point~$Q \in Y(k)$ the local parts $H^1_{Hdg, Q}$, $H^1_{dR, Q}$ should depend only on the ring $\wh{\mc O}_{X, Q} := (\pi_* \mc O_X)_Q \otimes_{\mc O_{Y, Q}} \wh{\mc O}_{Y, Q}$ (where $\wh{\mc O}_{Y, Q}$ denotes completion of the ring $\mc O_{Y, Q}$ with respect to the ideal $\mf m_{Y, Q}$).\\

The goal of this article is to propose a new way of computing the local parts, by establishing a connection with the Harbater--Katz--Gabber covers (in short: \emph{HKG-covers}).
The HKG-covers proved to be an important tool in the study of local actions and the deformation theory
of curves with automorphisms, see e.g. \cite{Bleher_Poonen_Chinburg_Auto_HKG},
\cite{Chinburg_Guralnick_Harbater_Oort_groups_lifting},
\cite{Chinburg_Guralnick_Harbater_local_lifting_for_actions}
\cite{Obus_local_lifting_problem}, \cite{Obus_Wewers_Cyclic_extensions},
\cite{Pop_OortConjecture} and~\cite{kontogeorgis_terezakis_new_obstruction}.
For any $Q \in Y(k)$, one may construct an HKG-cover $\ms X_Q \to \PP^1$ that approximates the cover $\pi : X \to Y$ locally over
$Q$, see below for a precise definition. It is natural to try to relate its cohomology with the postulated
local parts $H^1_{Hdg, Q}$, $H^1_{dR, Q}$ of the cohomologies of~$X$. Such a result would reduce investigation of cohomologies of $G$-covers
to HKG-covers. We show that this 
philosophy is correct for generic $p$-group covers.\\

Let $k$, $G$ and $\pi : X \to Y$ be as above.
Denote by~$g_Y$ the genus of $Y$ and by $B \subset Y(k)$ the branch locus of $\pi$.
For any ${P \in X(k)}$ denote by $G_{P, i}$ the $i$-th ramification group
at~$P$ and let:
\begin{equation*}
	d_{P}:= \sum_{i \ge 0} (\# G_{P, i} - 1), \quad 
	d_{P}' := \sum_{i \ge 1} (\# G_{P, i} - 1), \quad
	d_{P}'' := \sum_{i \ge 2} (\# G_{P, i} - 1).
\end{equation*}
Note that $d_P$ is the exponent of the different of $k(X)/k(Y)$ at $P$. We assume that the cover $\pi$ satisfies the following assumptions,
introduced in~\cite{Garnek_p_gp_covers}:
\begin{enumerate}[(A)]
	\item \label{enum:A} the stabilizer $G_P$ of $P$ in $G$ is a normal subgroup of $G$ for every $P \in X(k)$,
	
	\item \label{enum:B} there exists a function $z \in k(X)$ (a ``magical element'') satisfying $\ord_P(z) \ge -d_P'$
	for every $P \in X(k)$ and $\tr_{X/Y}(z) \neq 0$.
\end{enumerate}
Recall that a generic $p$-group cover satisfies~\ref{enum:A} and~\ref{enum:B} (cf. \cite[Theorem~1.5]{Garnek_p_gp_covers}). By the assumption~\ref{enum:A}, for $Q \in Y(k)$ we may denote $G_Q := G_P$ and $d_Q := d_P$
for any $P \in \pi^{-1}(Q)$.

Results of Harbater (cf. \cite{Harbater_moduli_of_p_covers}) and of Katz and Gabber (cf.~\cite{Katz_local_to_global}) imply that for any $G$-Galois algebra $\ms B$ over $k[[x]]$
there exists a unique $G$-cover $\ms X \to \PP^1$ ramified only over $\infty$
and such that there exists an isomorphism $\wh{\mc O}_{\ms X, \infty} \cong \ms B$ of $k[G]$-algebras.
The cover $\ms X \to \PP^1$ is called the \emph{Harbater--Katz--Gabber cover} associated to $(\ms B, G)$. 
Suppose now that $\pi : X \to Y$ is as above. Then $\wh{\mc O}_{X, Q}$
is a $G$-Galois algebra over $\wh{\mc O}_{Y, Q} \cong k[[x]]$ for any $Q \in Y(k)$. Denote by $\ms X_Q \to \PP^1$
the corresponding HKG-cover. Note that it might be disconnected. In fact,
$\ms X_Q = \bigsqcup_{G/G_Q} \ms X_Q^{\circ}$,
where $\ms X_Q^{\circ} \to \PP^1$ is a (connected) $G_Q$-HKG-cover.
We give now an example of computation of $\ms X_Q$.
\begin{Example} \label{ex:HKG_covers}
	Let $p = 2$, $k = \ol{\FF}_2$ and $G = \VV_4$, the Klein-four group. Let
	$Y$ be the elliptic curve with the affine equation $w^2 + w = u^3$ over the field $k$.
	Consider the $\VV_4$-cover $\pi : X \to Y$ given by the equations:
	\begin{align*}
		y_0^2 + y_0 = w^3 + \frac 1{w^7}, \qquad y_1^2 + y_1 = w^5 + \frac 1{w^7}.
	\end{align*}
	The cover $\pi$ is branched over the points $Q_1, Q_2 \in Y(k)$,
	where $Q_1 = (0, 0) \in Y(k)$ and $Q_2$ is the point at infinity of $Y$. We give now the equations of $\ms X_{Q_1}$ and $\ms X_{Q_2}$. Since $\ord_{Q_1}(w) = 3$,
	we have $w = x^{-3}$ for some element $x \in \wh{\mc O}_{Y, Q_1}$, $\ord_{Q_1}(x) = -1$. By Hensel's lemma, there exist $s_0, s_1 \in \wh{\mc O}_{Y, Q_1}$
	such that $s_0^2 + s_0 = x^{-9}$, $s_1^2 + s_1 = x^{-15}$. Let $z_0 := y_0 + s_0$, $z_1 := y_0 + y_1 + s_0 + s_1$.
	One easily checks that
	\[
	z_0^2 + z_0 = x^{21}, \qquad z_1^2 + z_1 = 0.
	\]
	These equations define $\ms X_{Q_1}$. Note that $\ms X_{Q_1} = \ms X_{Q_1}^{\circ} \sqcup \ms X_{Q_1}^{\circ}$,
	where $\ms X_{Q_1}^{\circ} : z_0^2 + z_0 = x^{21}$ is a $\ZZ/2$-cover of $\PP^1$.
	Similarly, using the fact that $\ord_{Q_2}(w^3 + \frac 1{w^7}) = -9$ and $\ord_{Q_2}(w^5 + \frac 1{w^7}) = -15$, one can show that $\ms X_{Q_2} = \ms X_{Q_2}^{\circ}$ is given by the equations:
	\begin{align*}
		z_0^2 + z_0 = x^9, \qquad
		z_1^2 + z_1 = x^{15}.
	\end{align*}
	Note that $\ms X_Q = \bigsqcup_{i = 1}^4 \PP^1$ is a trivial $\VV_4$-cover of $\PP^1$ for any $Q \in Y(k) \setminus \{ Q_1, Q_2 \}$.
\end{Example}
For any $k[G]$-module $V$, we write $V^{\vee}$ for the dual $k[G]$-module. Let $I_G := \{ \sum_{g \in G} a_g g \in k[G] : \sum_{g \in G} a_g = 0 \}$ be the augmentation ideal of the group~$G$. For any subgroup $H \le G$
we consider also the relative augmentation ideal $I_{G, H} := \Ind^G_H \, I_H$, which can be treated as a submodule of $k[G]$ (see Section~\ref{sec:p_gp}). Finally, the $k[G]$-module
$I_{X/Y}$ is defined~by:
\begin{equation*}
	I_{X/Y} := \ker \left( \sum : \bigoplus_{Q \in Y(k)} I_{G, G_Q} \to I_G \right).
\end{equation*}
The following is the main result of the paper.
\begin{Theorem} \label{thm:cohomology_of_G_covers}
	Keep the above assumptions. We have the following isomorphisms of $k[G]$-modules:
	\begin{align*}
		H^0(X, \Omega_X) &\cong k[G]^{\oplus g_Y} \oplus I_{X/Y} \oplus \bigoplus_{Q \in B} H^0(\ms X_Q, \Omega_{\ms X_Q}),\\
		H^1(X, \mc O_X) &\cong k[G]^{\oplus g_Y} \oplus I_{X/Y}^{\vee} \oplus \bigoplus_{Q \in B} H^1(\ms X_Q, \mc O_{\ms X_Q}),\\
		H^1_{dR}(X) &\cong k[G]^{\oplus 2 \cdot g_Y} \oplus I_{X/Y} \oplus I_{X/Y}^{\vee} 
		\oplus \bigoplus_{Q \in B} H^1_{dR}(\ms X_Q).
	\end{align*}
	Moreover, one can further decompose the local terms as follows:
	\begin{align*}
		H^0(\ms X_Q, \Omega_{\ms X_Q}) &\cong \Ind^G_{G_Q} H^0(\ms X_Q^{\circ}, \Omega_{\ms X_Q^{\circ}}),\\
		H^1(\ms X_Q, \mc O_{\ms X_Q}) &\cong \Ind^G_{G_Q} H^1(\ms X_Q^{\circ}, \mc O_{\ms X_Q^{\circ}}),\\
		H^1_{dR}(\ms X_Q) &\cong \Ind^G_{G_Q} H^1_{dR}(\ms X_Q^{\circ}).
	\end{align*}
\end{Theorem}
In the paper \cite{Garnek_p_gp_covers} we showed a statement similar to Theorem~\ref{thm:cohomology_of_G_covers}, but
with a different form of the local parts (cf. ibid, Theorem 1.1).
The new result surpasses it in two aspects. Firstly, the local parts in \cite{Garnek_p_gp_covers} depended not only on the local rings, but also on the element $z \in k(X)$
from the condition~\ref{enum:B}. The local parts given in Theorem~\ref{thm:cohomology_of_G_covers} depend only on the completed local rings,
as they come from cohomologies of $\ms X_Q$. Secondly, as the local parts in \cite{Garnek_p_gp_covers} were defined
as quotients of certain infinitely dimensional vector spaces, their computation for concrete examples was considerably challenging.
The local parts in Theorem~\ref{thm:cohomology_of_G_covers} can be computed by giving bases
of cohomologies of $\ms X_Q$, which seems much easier then giving bases of cohomologies of $X$. 
We present this strategy for Klein four covers in characteristic $2$, see below.
In order to prove Theorem~\ref{thm:cohomology_of_G_covers}, we use \cite[Theorem~1.1]{Garnek_p_gp_covers} together with a new description of the cohomologies of a HKG-cover (cf. Theorem~\ref{thm:cohomology_of_HKG_etale_algebras} in Section~\ref{sec:hkg}).\\

As an application of Theorem~\ref{thm:cohomology_of_G_covers}, we describe the equivariant structure of the de~Rham cohomology of Klein four covers of projective line in characteristic~$2$. Keep the above notation with $p = 2$ and $G = \VV_4 = \{ e, \sigma, \tau, \sigma \tau \}$. Recall that the indecomposable $k[\VV_4]$-modules are completely classified (cf. \cite{Basev_reps_Z2_Z2} or \cite[Appendix]{Bleher_Camacho_Holomorphic_differentials}). In the sequel we will need five indecomposable $k[\VV_4]$-modules apart from $k$ (i.e. the trivial representation) and $k[\VV_4]$,
cf. Table~\ref{tab:modules}. Note that in order to construct a $k[\VV_4]$-module of dimension $n$, we have to give
a pair of commuting square matrices of order at most~$2$ in $\Gl_n(k)$, corresponding to the action of~$\sigma$ and~$\tau$. 
\addtocounter{Theorem}{1}
\begin{table}[htbp]
	\label{tab:modules}	
	\caption{$k[\VV_4]$-modules used in the article.}
	\begin{tabular}{|c|c|c|c|}
		\hline
		Module & $\sigma$ & $\tau$ & Dual module \\
		\hline
		$N_{2, 0}$ & $\begin{pmatrix} 1 & 1 \\ 0 & 1 \end{pmatrix}$ & $\begin{pmatrix} 1 & 0 \\ 0 & 1 \end{pmatrix}$
		& $N_{2, 0}$\\
		$N_{2, 1}$ & $\begin{pmatrix} 1 & 0 \\ 0 & 1 \end{pmatrix}$ & $\begin{pmatrix} 1 & 1 \\ 0 & 1 \end{pmatrix}$
		& $N_{2, 1}$\\
		$N_{2, \infty}$ & $\begin{pmatrix} 1 & 1 \\ 0 & 1 \end{pmatrix}$ & $\begin{pmatrix} 1 & 1 \\ 0 & 1 \end{pmatrix}$
		& $N_{2, \infty}$ \\
		$M_{3, 1}$ & $\begin{pmatrix} 1 & 0 & 1 \\ 0 & 1 & 0 \\ 0 & 0 & 1 \end{pmatrix}$ & $\begin{pmatrix} 1 & 1 & 0 \\ 0 & 1 & 0 \\ 0 & 0 & 1 \end{pmatrix}$ & $M_{3, 2}$ \\
		$M_{3, 2}$ & $\begin{pmatrix} 1 & 0 & 1 \\ 0 & 1 & 0 \\ 0 & 0 & 1 \end{pmatrix}$ & $\begin{pmatrix} 1 & 0 & 0 \\ 0 & 1 & 1 \\ 0 & 0 & 1 \end{pmatrix}$ & $M_{3, 1}$\\
		\hline
	\end{tabular}
	\centering
\end{table}

The Klein four covers in characteristic~$2$ were studied in several articles, including \cite{Glass_Klein_covers}
and \cite{Glass_Pries_Klein_covers}.
As shown recently in \cite{Bleher_Camacho_Holomorphic_differentials}, there exist infinitely many isomorphism classes of indecomposable
$k[\VV_4]$-modules that may appear as a direct summand of the module of holomorphic differentials
of a $\VV_4$-cover in characteristic~$2$. As we will see, this
remains in stark contrast with the case of the de~Rham cohomology. 
\begin{Theorem} \label{thm:de_rham_of_klein_intro}
	Let $k$ be an algebraically closed field of characteristic~$2$.
	Suppose that $\pi : X \to Y$ is a $\VV_4$-cover of smooth projective
	curves over $k$ that satisfies the condition~\ref{enum:B} (this holds automatically for example
	if $Y = \PP^1$). There exist 7 isomorphism classes of indecomposable $k[\VV_4]$-modules that may appear
	as a direct summand of $H^1_{dR}(X)$:
	\[
		k, \, k[\VV_4], \,
		N_{2, 0}, \, N_{2, 1}, \, N_{2, \infty}, \, M_{3, 1} \, \textrm{ and } \, M_{3, 2}.
	\]
\end{Theorem}
In fact, we provide a more precise result. Namely, in Theorem~\ref{thm:de_rham_of_klein} we give a decomposition of $H^1_{dR}(X)$ into indecomposable summands
in terms of local invariants of the cover for curves satisfying the assumptions of Theorem~\ref{thm:de_rham_of_klein_intro}.
It seems that it would be much harder to obtain this
result using \cite[Theorem~1.1]{Garnek_p_gp_covers}.
\begin{Example} \label{ex:klein_cohomology_intro}
	Let $\pi : X \to Y$ be as in Example~\ref{ex:HKG_covers}. We compute now the equivariant structure of $H^1_{dR}(X)$.
	In order to compute $I_{X/Y}$, note that $I_{G, G_{Q_2}} = I_G$. Hence:
	\begin{align*}
		I_{X/Y} \cong \ker(I_{G, G_{Q_1}} \oplus I_{G, G_{Q_2}} \to I_G) \cong I_{G, G_{Q_1}}.
	\end{align*}
	By \cite[Corollary~1.2]{Garnek_p_gp_covers}, $H^1_{dR}(\ms X_{Q_1}^{\circ}) \cong I_{G_{Q_1}}^{\oplus 20}$ and hence:
	\[
		H^1_{dR}(\ms X_{Q_1}) \cong \Ind^{G}_{G_{Q_1}} H^1_{dR}(\ms X_{Q_1}^{\circ}) \cong I_{G, G_{Q_1}}^{\oplus 20}.
	\]
	One shows that $I_{G, G_{Q_1}} \cong I_{G, G_{Q_1}}^{\vee} \cong N_{2, \infty}$, see the proof of Lemma~\ref{lem:IXY_of_Klein_covers}.
	By computing an explicit basis of the de Rham cohomology of $\ms X_{Q_2}$ (see Proposition~\ref{prop:dR_of_Klein_HKG}), one proves that:
	\begin{align*}
		H^1_{dR}(\ms X_{Q_2}) \cong N_{2, 0}^{\oplus 6} \oplus M_{3, 1}^{\oplus 4}
		\oplus M_{3, 2}^{\oplus 4}.
	\end{align*}
	It turns out that $X$ satisfies the condition~\ref{enum:B}, cf. Example~\ref{ex:elliptic_curve_ctn}.
	Hence by Theorem~\ref{thm:cohomology_of_G_covers}:
	\begin{align*}
		H^1_{dR}(X) &\cong k[\VV_4]^{\oplus 2} \oplus I_{G, Q_2} \oplus I_{G, Q_2}^{\vee} \oplus H^1_{dR}(\ms X_{Q_1}) \oplus H^1_{dR}(\ms X_{Q_2})\\
		&\cong k[\VV_4]^{\oplus 2} \oplus N_{2, \infty}^{\oplus 22} \oplus N_{2, 0}^{\oplus 6} \oplus M_{3, 1}^{\oplus 4}
		\oplus M_{3, 2}^{\oplus 4}.
	\end{align*}
\end{Example}
\subsection*{Outline of the paper}
In Section~\ref{sec:prelim} we discuss preliminaries on algebraic curves. Section~\ref{sec:hkg} is devoted to a new description of the cohomologies of HKG-covers. In Section~\ref{sec:p_gp} we prove Theorem~\ref{thm:cohomology_of_G_covers}. Section~\ref{sec:klein_covers} recalls basic facts concerning $\VV_4$-covers. In Section~\ref{sec:de_Rham_of_Klein}, we
compute the equivariant structure of the de Rham cohomologies of Klein four covers, assuming a result on the cohomologies of $\VV_4$-HKG covers (Proposition~\ref{prop:dR_of_Klein_HKG}). Proposition~\ref{prop:dR_of_Klein_HKG} is proven in the last section.
\subsection*{Acknowledgements}
The author wishes to express his thanks to Wojciech Gajda and Piotr Achinger for many stimulating conversations.
The author gratefully acknowledges also Bartosz Naskręcki and Aleksandra Kaim-Garnek,
whose comments helped to considerably improve the exposition of the paper.
Some of the results in this paper were obtained during the author's stays in Fort Collins and in Paris. The author would like to thank his hosts, Rachel Pries and
Marc Hindry, for their warm hospitality. The author was supported by the research grant SONATINA 6 "The de~Rham cohomology of $p$-group covers" UMO-2022/44/C/ST1/00033
awarded by National Science Centre, Poland.

\section{Preliminaries} \label{sec:prelim}
In this section we introduce notation concerning algebraic curves and recall basic facts used throughout the paper.
For an arbitrary smooth projective curve $Y$ over a field $k$ we denote by $k(Y)$ the function field of $Y$ and by $g_Y$ its genus.
Also, we write $\ord_Q(f)$ for the order of vanishing of a function $f \in k(Y)$ at a point $Q \in Y(k)$.
Let $\mf m_{Y, Q}^n := \{ f \in k(Y) : \ord_Q(f) \ge n \}$ for any $n \in \ZZ$. To simplify notation, we write
$\Omega_Y$, $\Omega_{k(Y)}$ and $H^1_{dR}(Y)$ instead of $\Omega_{Y/k}$, $\Omega_{k(Y)/k}$ and $H^1_{dR}(Y/k)$.
We often identify a finite set $S \subset Y(k)$ with a reduced divisor in $\Divv(Y)$.
Thus e.g. $\Omega_Y(S)$ will denote the sheaf of logarithmic differential
forms with poles in $S$. In the sequel we often use residues of differential forms,
see e.g. \cite[Remark III.7.14]{Hartshorne1977} for relevant facts.\\

Let $G$ be a finite group and $\pi : X \to Y$ be a finite separable $G$-cover of smooth projective curves over a field $k$.
We identify $\Omega_{k(Y)}$ with a submodule of $\Omega_{k(X)}$ and
$k(Y)$ with a subfield of $k(X)$. We denote the ramification index of $\pi$ at $P \in X(k)$ by $e_{X/Y, P}$ and
by $G_{P, i}$ the $i$-th ramification group of $\pi$ at $P$, i.e.
\[
G_{P, i} := \{ \sigma \in G : \sigma(f) \equiv f \mod{\mf m_P^{i+1}} \quad \textrm{ for every } f \in \mc O_{X, P} \}.
\]
Also, we use the following notation:
\begin{align*}
d_{X/Y, P}:= \sum_{i \ge 0} (\# G_{P, i} - 1), \quad
d_{X/Y, P}' := \sum_{i \ge 1} (\# G_{P, i} - 1), \quad
d_{X/Y, P}'' := \sum_{i \ge 2} (\# G_{P, i} - 1)
\end{align*}
($d_{X/Y, P}$ is the exponent of the different of $k(X)/k(Y)$ at $P$, cf. \cite[Proposition~IV.\S 1.4]{Serre1979}).
Recall that for any $P \in X(k)$ and $\omega \in \Omega_{k(Y)}$:
\begin{equation} \label{eqn:valuation_of_diff_form}
\ord_P(\omega) = e_{X/Y, P} \cdot \ord_{\pi(P)}(\omega) + d_{X/Y, P}
\end{equation}
(see e.g. \cite[Proposition IV.2.2~(b)]{Hartshorne1977}). For any sheaf $\mc F$ on $X$ and $Q \in Y(k)$ we abbreviate $(\pi_* \mc F)_Q$ to $\mc F_Q$
and $(\pi_* \mc F)_Q \otimes_{\mc O_{Y, Q}} \wh{\mc O}_{Y, Q}$ to $\wh{\mc F}_Q$. We write briefly $\tr_{X/Y}$ for the trace
\[
\tr_{k(X)/k(Y)} : k(X) \to k(Y).
\]
Note that it induces a map
\[
\Omega_{k(X)} \cong k(X) \otimes_{k(Y)} \Omega_{k(Y)} \to \Omega_{k(Y)},
\]
which we also denote by $\tr_{X/Y}$.
Recall that for any $\eta \in \Omega_{k(X)}$ and $Q \in Y(k)$:
\begin{equation} \label{eqn:residue_and_trace}
	\sum_{P \in \pi^{-1}(Q)} \res_P(\eta) = \res_Q(\tr_{X/Y}(\eta))
\end{equation}
(see \cite[Proposition 1.6]{Hubl_residual_representation} or \cite[p.~154, $(R_6)$]{Tate_residues_differentials_curves}).

In the sequel we use also the following description of cohomologies on a smooth projective curve $Y$. Let $S \subset Y(k)$ be a finite non-empty set. Write $U := Y \setminus S$ and let $\eta$ be the generic point of $Y$. For any locally free sheaf $\mc F$ of finite rank we have a natural isomorphism (cf. \cite[Lemma~5.2]{Garnek_p_gp_covers}):
\begin{align} \label{eqn:formula_for_H1}
	H^1(Y, \mc F) &\cong \coker(\mc F(U) \to \bigoplus_{Q \in S} \mc F_{\eta}/\mc F_{Q}),
\end{align}
Similarly, if $\mc F^{\bullet} = (\mc F^0 \stackrel{d}{\longrightarrow} \mc F^1)$ is a cochain
complex of locally free $\mc O_Y$-modules of finite rank with a $k$-linear differential
then (cf. \cite[Lemma~6.2]{Garnek_p_gp_covers}):
\begin{align}\label{eqn:formula_for_H1_dR}
	\HH^1(Y, \mc F^{\bullet}) &\cong Z^1_S(\mc F^{\bullet})/B^1_S(\mc F^{\bullet}),
\end{align}
where:
\begin{align*}
	Z^1_{S}(\mc F^{\bullet}) &:= \{ (\omega, (h_Q)_{Q \in S} ) : \omega \in \mc F^1(U), 
	h_Q \in \mc F^0_{\eta}, \, \omega - d h_Q \in \mc F^1_Q \quad \forall {Q \in S} \},\\
	B^1_{S}(\mc F^{\bullet}) &:= \{ (dh, (h + h_Q)_{Q \in S} ) : h \in \mc F^0(U), \,
	h_Q \in \mc F^0_Q \quad \forall {Q \in S} \}.
\end{align*}
Suppose now that $Y = \bigsqcup_{i = 1}^r Y_i$ is a disjoint union of smooth projective curves.
Note that $H^0(Y, \Omega_Y) = \bigoplus_{i = 1}^r H^0(Y_i, \Omega_{Y_i})$, $H^1(Y, \mc O_Y) = \bigoplus_{i = 1}^r H^1(Y_i, \mc O_{Y_i})$,
etc. By abuse of notation, we will denote by $k(Y)$ the total fraction field of $Y$, i.e. $k(Y) := \bigoplus_{i = 1}^r k(Y_i)$.

\section{Cohomology of HKG-covers} \label{sec:hkg}
Let $k$ and $G$ be as in Section~\ref{sec:intro}. Suppose $\ms B$ is a $k$-algebra with a $k$-linear
action of $G$ such that $\ms A := \ms B^G$ equals $k[[x]]$. Results of Harbater (cf. \cite{Harbater_moduli_of_p_covers}) and of Katz and Gabber (cf.~\cite{Katz_local_to_global}) imply that there exists a unique $G$-cover $\pi : \ms X \to \PP^1$ ramified only over $\infty$
and such that there exists an isomorphism $\wh{\mc O}_{X, Q} \cong \ms B$ of $k[G]$-algebras. We call $\pi : \ms X \to \PP^1$ the \emph{Harbater--Katz--Gabber cover} (in short: \emph{HKG-cover}) associated to $(\ms B, G)$.
The following natural question arises:
\begin{Question}
	How to describe the invariants of the curve $\ms X$ (e.g. its Hodge and de~Rham cohomologies) in terms of $\ms B$?
\end{Question}
This problem was studied in several papers, see e.g. \cite{Kontogeorgis_Tsouknidas_Cohomological_HKG}, \cite{Kontogeorgis_Tsouknidas_generating_set_HKG}
and \cite{Karanikolopoulos_Kontogeorgis_Autos}. The listed articles
give for instance a way of computing $H^0(\ms X, \Omega_{\ms X})$. However, this description depends on the choice of the Artin--Schreier tower and thus is impractical for our purposes. We offer an alternative approach
using certain distinguished elements in the total fraction field of $\ms B$.\\

We start by introducing the needed notation. Let $\ms L$ be be the total fraction field
of $\ms B$. Then $\ms B = \prod_{i = 1}^r k[[t_i]]$ and $\ms L = \prod_{i=1}^r k((t_i))$. Similarly, let $\ms K = k((x))$ be the fraction field of $\ms A$ and $\mf m_{\ms A} := (x)$ be the
maximal ideal of $\ms A$. Let $\mf m_i = (t_i)$ be the unique maximal ideal of $k[[t_i]]$; we will identify it with the corresponding ideal of $\ms B$. Also, let $\mf m_{\ms B} := \bigcap_{i=1}^r \mf m_i$ be the Jacobson radical of $\ms B$
and let $\ord_{\mf m_i} : \ms L \to \ZZ \cup \{ \infty \}$ be the composition of the valuation $k((t_i)) \to \ZZ \cup \{ \infty \}$
associated to $\mf m_i$ with the surjection $\ms L \twoheadrightarrow k((t_i))$. We denote by $G_j(\mf m_i/\mf m_{\ms A})$ the $j$th ramification
group of $\mf m_i$ over $\mf m_{\ms A}$. Note that in particular, $G_0(\mf m_i/\mf m_{\ms A}) \cong \Gal(k((t_i))/\ms K)$ is the decomposition group of
$\mf m_i$ over $\mf m_{\ms A}$. From now on we make the following assumption:
\begin{center}
	$(*)$ \quad $G_0(\mf m_1/\mf m_{\ms A})$ is a normal subgroup of $G$. 
\end{center}
This implies in particular, that all the extensions $k((t_1))/\ms K, \ldots, k((t_r))/\ms K$ are isomorphic and $G_0(\mf m_1/\mf m_{\ms A})$ is the decomposition group
of $\mf m_i$ over $\mf m_{\ms A}$ for every $i = 1, \ldots, r$. Thus we may identify $\ms L$ with $\prod_{G/G_0} k((t_1))$. We denote also from now on $G_i := G_i(\mf m_1/\mf m_{\ms A})$. 
Let $e = e(\mf m_i/\mf m_{\ms A}) := \# G_0$ be the ramification index of
$\mf m_i$ over $\mf m_{\ms A}$. Similarly, let
\begin{equation*}
	d := \sum_{i \ge 0} (\# G_i - 1), \quad 
	d' := \sum_{i \ge 1} (\# G_i - 1), \quad
	d'' := \sum_{i \ge 2} (\# G_i - 1).
\end{equation*}
For any topological $k$-algebra~$C$, denote by
$\Omega_{C}$ the module of continuous $k$-linear K\"{a}hler differentials.
In particular $\Omega_{\ms A} = k[[x]] \, dx$, $\Omega_{\ms K} = k((x)) \, dx$, $\Omega_{\ms B} = \bigoplus_{i = 1}^r \ms B \, dt_i$,
$\Omega_{\ms L} = \bigoplus_{i = 1}^r \ms L \, dt_i$. Similarly, let
\[
\Omega_{\ms A}^{\log} = \frac{1}{x} \Omega_{\ms A}, \quad \Omega_{\ms B}^{\log} = \frac{1}{t_1 \cdot t_2 \cdot \ldots \cdot t_r} \Omega_{\ms B}.
\]
be the logarithmic K\"{a}hler differentials.\\

Write $R = \pi^{-1}(\infty)$ for the ramification set
of $\pi$ and denote $U := \PP^1 \setminus \{ \infty \}$, $V :=  \pi^{-1}(U)$. For a future use, note that:
\begin{equation} \label{eqn:HKG_is_Ind_HKG}
	\ms X = \Ind^G_{G_0} \ms X^{\circ},
\end{equation}
where $\ms X^{\circ} \to \PP^1$ is the HKG-cover associated to $(k[[t_1]], G_Q)$ and $\Ind^G_{G_0} \ms X^{\circ}$ is the disjoint union $\bigsqcup_{G/G_0} \ms X^{\circ}$ with the natural action of $G$. The Riemann--Hurwitz formula (cf. \cite[Corollary IV.2.4]{Hartshorne1977}) for the cover $\ms X^{\circ} \to \PP^1$ implies that $g_{\ms X^{\circ}} = \frac 12 d''$. \\

We say that $z \in \ms L$ is a \emph{magical element} for the extension $\ms L/\ms K$, if
$\ord_{\mf m_i}(z) \ge -d'$ for every $i = 1, \ldots, r$ and $\tr_{\ms L/\ms K}(z) \in \ms A^{\times}$.
In this context, we denote by $z^{\vee}$ the dual element of $z$ with respect to the trace pairing. In other words,
$z^{\vee}$ is defined by the following equalities for all $g_1, g_2 \in G$:
\[
	\tr_{\ms L/\ms K}(g_1(z) \cdot g_2(z^{\vee})) =
	\begin{cases}
		1, & g_1 = g_2,\\
		0, & \textrm{otherwise.}
	\end{cases}
\]

The following is the main result of this section.
\begin{Theorem} \label{thm:cohomology_of_HKG_etale_algebras}
	Keep the above notation. In particular, let $\ms X \to \PP^1$ be the HKG-cover corresponding to
	$(\ms B, G)$. Suppose that $z \in \ms L$ is a magical element.
	Then we have the following isomorphisms of $k[G]$-modules:
	\begin{align*}
		H^0(\ms X, \Omega_{\ms X}) &\cong \frac{\Omega_{\ms B}^{\log}}{\bigoplus_{g \in G} g^*(z) \Omega_{\ms A}^{\log}},\\
		H^1(\ms X, \mc O_{\ms X}) &\cong \frac{\bigoplus_{g \in G} g^*(z^{\vee}) \mf m_{\ms A}}{\mf m_{\ms B}},\\
		H^1_{dR}(\ms X) &\cong \left\{ (\omega, \nu) \in \frac{\Omega_{\ms L}}{\bigoplus_{g \in G} g^*(z) \Omega_{\ms A}^{\log}} \times \frac{\bigoplus_{g \in G} g^*(z^{\vee}) \mf m_{\ms A}}{\mf m_{\ms B}}:
		\quad \omega - d\nu \in \Omega_{\ms B}^{\log} \right\}.
	\end{align*}
\end{Theorem}
\noindent For future use, recall that for any $f \in \ms L$:
\begin{equation} \label{eqn:valuation_of_trace}
	\tr_{\ms L/\ms K}(f) \in \mf m_{\ms A}^{\alpha},
\end{equation}
where $\alpha := \min_{i=1, \ldots, r} \lfloor (\ord_{\mf m_i}(f)+d)/e \rfloor$.
For the proof see \cite[Lemma 1.4 (b)]{Kock_galois_structure}. 

We list now some properties of magical elements. 
Note that if $z \in \ms L$ is a magical element for $\ms L/\ms K$ then there exists $1 \le i \le r$ such that $\ord_{\mf m_i}(z) = -d'$. Indeed, if we would have $\ord_{\mf m_i}(z) > -d'$ for $i = 1, \ldots, r$, then $\tr_{\ms L/\ms K}(z) \in \mf m_{\ms A}$ by~\eqref{eqn:valuation_of_trace}, which would yield a contradiction.
We prove now a partial converse of this fact.
\begin{Lemma} \label{lem:criterion_magical}
	Keep the above assumptions. If $z \in \ms L$ satisfies $\ord_{\mf m_i}(z) \ge -d'$
	for all $i = 1, \ldots, r$ and equality holds for precisely one $i$, then $z$ is a magical element for $\ms L/\ms K$.
\end{Lemma}
\begin{proof}
	Let $z = (z_1, \ldots, z_r) \in \ms L = \prod_{i = 1}^r k((t_i))$. Suppose that
	$\ord_{\mf m_1}(z) = -d'$ and $\ord_{\mf m_i}(z) > -d'$ for $i = 2, \ldots, r$.
	Note that by~\eqref{eqn:valuation_of_trace} we have $\tr_{\ms L/\ms K}(z_1) \in \ms A$ and $\tr_{\ms L/\ms K}(z_i) \in \mf m_{\ms A}$ for $i = 2, \ldots, r$.
	Let $\ms K = \ms K_0 \subset \ms K_1 \subset \ldots \subset \ms K_n = k((t_1))$ be a tower of $\ZZ/p$-extensions.
	Note that $\ms K_i$ is a field with a discrete valuation $\ord_{\ms K_i}$.
	Let $\ms K_{i+1}/\ms K_i$ be given by the Artin--Schreier equation:
	\[
	y_i^p - y_i = h_i \in \ms K_i \quad \textrm{ for } i = 0, \ldots, n-1.
	\]
	Note that $\ord_{\ms K_i}(h_i) < 0$ for every $i = 0, \ldots, n-1$, since otherwise $\ms K_{i+1}$ wouldn't be a field.
	Without loss of generality, we may assume that $p \nmid \ord_{\ms K_i}(h_i)$ for every $i = 1, \ldots, r$.
	In this context one often says that $y_1, \ldots, y_n$ are the Artin--Schreier generators for $k((t_1))/\ms K$ in standard form.
	Then \cite[Lemmas~7.2 and~8.1]{Garnek_p_gp_covers} show that $z' := y_1^{p-1} \cdot \ldots \cdot y_n^{p-1}$
	satisfies $\tr_{\ms L/\ms K}(z') = \pm 1$ and $\ord_{\mf m_1}(z') = -d'$. Thus, for some $c \in k^{\times}$, $\ord_{\mf m_1}(z_1 - c z') > -d'$. But this implies by~\eqref{eqn:valuation_of_trace} that
	$\tr_{\ms L/\ms K}(z_1 - c z') \in \mf m_{\ms A}$. Therefore
	$\tr_{\ms L/\ms K}(z_1) \equiv \pm c \not \equiv 0 \pmod{\mf m_{\ms A}}$ and $\tr_{\ms L/\ms K}(z) = \sum_{i = 1}^r \tr_{k((t_i))/\ms K}(z_i) \in \ms A \setminus \mf m_{\ms A} = \ms A^{\times}$. 
	This ends the proof.
\end{proof}
\begin{Corollary} \label{cor:existence_of_magical}
	Every \'{e}tale $G$-Galois algebra $\ms L/\ms K$ satisfying $(*)$ has a magical element.
\end{Corollary}
\begin{proof}
	Let $z := (t_1^{-d'}, 0, \ldots, 0) \in \ms L = \prod_{i = 1}^r k((t_i))$.
	Then $z$ is a magical element for $\ms L/\ms K$ by Lemma~\ref{lem:criterion_magical}.
\end{proof}
In case when $\ms L$ is a field, the definition of a magical element simplifies.
\begin{Corollary} \label{cor:magical_for_field}
	Suppose that $G_0 = G$ (i.e. $\ms L = k((t_1))$). Then
	$z \in \ms L$ is a magical element for $\ms L/\ms K$ if and only if $\ord_{\mf m_1}(z) = -d'$.
\end{Corollary}
\begin{proof}
	Lemma~\ref{lem:criterion_magical} implies that if
	$\ord_{\mf m_1}(z) = -d'$, then $\tr_{\ms L/\ms K}(z) \in \ms A^{\times}$. 
	For the converse implication, note that if $\ord_{\mf m_1}(z) > -d'$,
	then $\tr_{\ms L/\ms K}(z) \in \mf m_{\ms A}$ by \eqref{eqn:valuation_of_trace}.
\end{proof}
The following result will be crucial in the proof of Theorem~\ref{thm:cohomology_of_HKG_etale_algebras}.
\begin{Proposition} \label{prop:key_fact_omega_L}
	Keep the assumptions of Theorem~\ref{thm:cohomology_of_HKG_etale_algebras}.
	Then:
	\[
		\Omega_{\ms L} = \bigoplus_{g \in G} g^*(z) \Omega_{\ms A}^{\log} \oplus H^0(V, \Omega_{\ms X}).
	\]
\end{Proposition}
\noindent We prove Proposition~\ref{prop:key_fact_omega_L} by induction on $\# G$.
Before the proof we need two auxiliary results.
\begin{Lemma} \label{lem:zH_is_magical}
	Keep the assumptions of Theorem~\ref{thm:cohomology_of_HKG_etale_algebras}. Then for any $H \unlhd G$,
	the extension $\ms L^H / \ms K$ satisfies the condition~$(*)$ and has 
	$z_H := \tr_{\ms L/\ms L^H}(z)$ as a magical element.
\end{Lemma}
\begin{proof}
	Note that the maximal ideals of $\ms B^H$ are $\mf m_i^H$ for $i = 1, \ldots, r$. Since
	$G_0 \unlhd G$, we have $G_0(\mf m_i^H/\mf m_{\ms A}) = G_0 H/H \unlhd G/H$. Thus
	$\ms L^H/\ms K$ also satisfies the condition~$(*)$. By the transitivity of the different (cf. \cite[\S III.4,
	Proposition~8]{Serre1979}):
	\[
		d'(\mf m_1/\mf m_{\ms A}) = d'(\mf m_1/\mf m_1^H) + e(\mf m_1/\mf m_1^H) \cdot d'(\mf m_1^H/\mf m_{\ms A}).
	\]
	By~\eqref{eqn:valuation_of_trace} we see that for $i = 1, \ldots, r$:
	\begin{align*}
		\ord_{\mf m_i^H}(z_H) \ge \left\lfloor \frac{-d'(\mf m_1/\mf m_{\ms A}) + d(\mf m_1/\mf m_1^H)}{e(\mf m_1/\mf m_1^H)} \right\rfloor = -d'(\mf m_1^H/\mf m_{\ms A}).
	\end{align*}
	Moreover, $\tr_{\ms L^H/\ms K}(z_H) = \tr_{\ms L/\ms K}(z) \in \ms A^{\times}$. This ends the proof.
\end{proof}
\begin{Lemma} \label{lem:H0VX_is_free_kG_module}
	$H^0(V, \Omega_{\ms X})$ is a free $k[G]$-module.
\end{Lemma}
\begin{proof}
	Let $V^{\circ}$ be the preimage of $\PP^1 \setminus \{\infty\}$ on $\ms X^{\circ}$.
	Using Riemann--Roch theorem, one sees that $H^0(\ms X^{\circ}, \mc O_{\ms X^{\circ}}(d' \cdot R)) \neq H^0(\ms X^{\circ}, \mc O_{\ms X^{\circ}}((d' - 1) \cdot R))$, since $\deg (d' \cdot R) = d' > 2g_{\ms X^{\circ}} - 2 = d'' - 2$.
	Let $z \in H^0(\ms X^{\circ}, \mc O_{\ms X^{\circ}}(d' \cdot R)) \setminus H^0(\ms X^{\circ}, \mc O_{\ms X^{\circ}}((d' - 1) \cdot R))$. Then $z$ is regular on~$V^{\circ}$ and $\tr_{\ms X^{\circ}/\PP^1}(z) \neq 0$ by Corollary~\ref{cor:magical_for_field}. Therefore $z$ is a normal element for $k((t_1))/\ms K$ by \cite[Proposition~3.1]{Garnek_p_gp_covers} and
	$\pi_* \Omega_{\ms X^{\circ}}|_U = \bigoplus_{g \in G_0} g^*(z) \Omega_{\PP^1}|_U$ as subsheaves of the constant sheaf $\pi_* \Omega_{k(\ms X^{\circ})}$ on $\PP^1$. Thus:
	\begin{align*}
		H^0(V, \Omega_{\ms X}) = \Ind^G_{G_0} H^0(V^{\circ}, \Omega_{\ms X^{\circ}}) = \Ind^G_{G_0} \bigoplus_{g \in G_0} g^*(z) H^0(U, \Omega_{\PP^1}) \cong k[G] H^0(U, \Omega_{\PP^1}).
	\end{align*}
\end{proof}
\begin{proof}[Proof of Proposition~\ref{prop:key_fact_omega_L}]
	We prove this by induction on $\# G$. Denote:
	\begin{align*}
		S_G(z) := \bigoplus_{g \in G} g^*(z) \Omega_{\ms A}^{\log}.
	\end{align*}
	For $\# G = 1$ this is true, as then $z \in k^{\times}$ and Proposition~\ref{prop:key_fact_omega_L} comes down to the equality:
	\[
	k((t)) \, dt = \frac 1t k[[t]] \, dt \oplus k[1/t] \frac{dt}{t^2}.
	\]
	Suppose now that $\# G > 1$. Let $H = \langle h \rangle \le G$ be a central subgroup of order~$p$ (it exists by \cite[Chapter~6, Theorem~1~(1)]{DummitFoote2004}).
	Denote $\ms X' := \ms X/H$ and $G' = G/H = \{ g_1 H, \ldots, g_s H \}$. Then $\ms X' \to \PP^1$ is the HKG-cover
	associated to $(\ms B^H, G')$.
	Suppose to the contrary that $\omega \in S_G(z) \cap H^0(V, \Omega_{\ms X})$, $\omega \neq 0$. Let $i \ge 0$ be the largest number for which
	$(h - 1)^i \cdot \omega \neq 0$  (note that $(h-1)^p = 0$, thus there exists such a number). 
	Denote $\omega' := (h - 1)^i \cdot \omega$. Then $(h-1) \cdot \omega' = 0$, i.e $\omega' \in \Omega_{k(\ms X')}$.
	On the other hand, $\omega' \in H^0(V, \Omega_{\ms X})$ clearly implies $\omega' \in H^0(V', \Omega_{\ms X'})$, where $V'$ is the image of $V$ through $\ms X \to \ms X'$. Let $z_H := \tr_{\ms L/\ms L^H}(z)$. Note that
	$z_H$ is a magical element for the extension $\ms L^H/\ms K$ by Lemma~\ref{lem:zH_is_magical}.
	We show now that $\omega' \in S_{G'}(z_H)$. Let $\omega' = \sum_{g \in G} g^*(z) \omega'_g$, where $\omega'_g \in \Omega_{\ms A}^{\log}$ for every $g \in G$. The condition $h(\omega') = \omega'$ implies that $\omega_{g h}' = \omega_g'$ for any $g \in G$. 
	Thus we can define $\omega_j' := \omega_g'$ for any $g \in g_j H$. We have:
	\[
		\omega' = \sum_{g \in G} g^*(z) \omega_g' = \sum_{j = 1}^s \omega_j' \sum_{g \in g_j H} g^*(z)
		= \sum_{j = 1}^s g_j^*(z_H) \omega_j' \in S_{G'}(z_H).
	\]
	This shows that $\omega' \in S_{G'}(z_H) \cap H^0(V', \Omega_{\ms X'})$. But by induction hypothesis,
	$S_{G'}(z_H) \cap H^0(V', \Omega_{\ms X'}) = \{ 0 \}$. Thus $S_{G}(z) \cap H^0(V, \Omega_{\ms X}) = \{ 0 \}$.\\
	Denote:
	\[
		\mc Q := \frac{\Omega_{\ms L}}{S_G(z) \oplus H^0(V, \Omega_{\ms X})}.
	\]
	We'll show now that $\mc Q = \{ 0 \}$. Suppose to the contrary that $\omega \in \Omega_{\ms L}$ has non-zero image in $\mc Q$.
	Let $i \ge 0$ be the largest number such that $(h - 1)^i \cdot \omega \neq 0$ in $\mc Q$ and denote $\omega' := (h - 1)^i \cdot \omega$.
	Then $\omega' \in \mc Q^H$. Using the long exact sequence of group cohomology for
	\[
	0 \to S_G(z) \oplus H^0(V, \Omega_{\ms X}) \to \Omega_{\ms L} \to \mc Q \to 0
	\]
	we obtain:
	\begin{align}
		0 &\to S_G(z)^H \oplus H^0(V, \Omega_{\ms X})^H \to \Omega_{\ms L}^H \to \mc Q^H \nonumber \\
		&\to H^1(H, S_G(z) \oplus H^0(V, \Omega_{\ms X})). \label{eqn:long_exact_invariants}
	\end{align}
	Note that $H^1(H, S_G(z)) = 0$, since $S_G(z)$ is a free $k[G]$-module and hence also a free $k[H]$-module.
	On the other hand, using Lemma~\ref{lem:H0VX_is_free_kG_module} we see that $H^1(H, H^0(V, \Omega_{\ms X})) = 0$.
	Thus~\eqref{eqn:long_exact_invariants} implies that
	$\Omega_{\ms L^H} = \Omega_{\ms L}^H \to \mc Q^H$ is a surjection. Therefore, without loss of generality we may assume that
	$\omega' \in \Omega_{\ms L^H}$. By induction hypothesis, $\omega' \in S_{G'}(z_H) \oplus H^0(V', \Omega_{\ms X'})$.
	But this implies that $\omega' = 0$ in $\mc Q$! Indeed, $H^0(V', \Omega_{\ms X'}) \subset H^0(V, \Omega_{\ms X})$
	and if $\eta = \sum_{j = 1}^s g_j^*(z_H) \eta_j \in S_{G'}(z_H)$ for $\eta_j \in \Omega_{\ms A}^{\log}$, then 
	\[
	\eta = \sum_{j=1}^s \eta_j \sum_{l = 0}^{p-1} (g_j h^l)^*(z) \in S_G(z),
	\]
	i.e. $S_{G'}(z_H) \subset S_G(z)$
	Contradiction ends the proof.
\end{proof}
\begin{proof}[Proof of Theorem~\ref{thm:cohomology_of_HKG_etale_algebras}]
	Denote the right hand sides in Theorem~\ref{thm:cohomology_of_HKG_etale_algebras} by $H^0(z)$, $H^1(z)$, $H^1_{dR}(z)$
	respectively.
	Let $S_G(z)$ be defined as in the proof of Proposition~\ref{prop:key_fact_omega_L} and denote:
	\begin{align*}
		S_G^{\vee}(z) := \bigoplus_{g \in G} g^*(z^{\vee}) \mf m_{\ms A}.
	\end{align*}
	The quotients in Theorem~\ref{thm:cohomology_of_HKG_etale_algebras} are well-defined. Indeed, \cite[Lemma~3.2]{Garnek_p_gp_covers} implies that
	$S_G(z) \subset \Omega_{\ms B}^{\log}$ and \cite[Lemma~3.4]{Garnek_p_gp_covers} yields $\mf m_{\ms B} \subset S_G^{\vee}(z)$.\\
	Proposition~\ref{prop:key_fact_omega_L} implies that
	\[
		\Omega_{\ms B}^{\log} = S_G(z) \oplus H^0(\ms X, \Omega_{\ms X}(R)),
	\]
	since $S_G(z) \subset \Omega_{\ms B}^{\log}$ and $H^0(V, \Omega_{\ms X}) \cap \Omega_{\ms B}^{\log} = H^0(\ms X, \Omega_{\ms X}(R))$. Moreover, the residue theorem (cf. \cite[Corollary after Theorem 3]{Tate_residues_differentials_curves}) implies that $H^0(\ms X, \Omega_{\ms X}(R)) = H^0(\ms X, \Omega_{\ms X})$.
	Thus we obtain the desired isomorphism:
	\[
		H^0(\ms X, \Omega_{\ms X}) \cong \frac{\Omega_{\ms B}^{\log}}{S_G(z)}.
	\]
	In the sequel we will use~\eqref{eqn:formula_for_H1} and~\eqref{eqn:formula_for_H1_dR} for the sheaf $\pi_* \mc O_{\ms X}$
	and the complex $\pi_* \Omega_{\ms X}^{\bullet}$ on $\PP^1$ with $S = \{ \infty \}$. Note that since $\pi$ is affine, $H^1(\ms X, \mc O_{\ms X}) \cong H^1(\PP^1, \pi_* \mc O_{\ms X})$
	and $H^1_{dR}(\ms X) \cong \HH^1(\PP^1, \pi_* \Omega_{\ms X}^{\bullet})$. Moreover, if we denote by $\eta$ the generic point of $\PP^1$,
	then $(\pi_* \mc O_{\ms X})_{\eta} \cong k(\ms X)$.
	The pairing
	\[
		k(\ms X) \times \Omega_{k(\ms X)} \to k, \quad (f, \omega) := \res_{\infty}(\tr_{X/\PP^1}(f \cdot \omega))
	\]
	induces both the Serre's duality $H^0(\ms X, \Omega_{\ms X})^{\vee} \cong H^1(\ms X, \mc O_{\ms X})$ and the duality between $\Omega_{\ms B}^{\log}/S_G(z)$ and
	$S_G^{\vee}(z)/\mf m_{\ms B}$ (cf. \cite[Lemma~3.5~(1)]{Garnek_p_gp_covers}). This proves that
	\[
		H^1(\ms X, \mc O_{\ms X}) \cong S_G^{\vee}(z)/\mf m_{\ms B}.
	\]
	Recall that $H^0(\ms X, \Omega_{\ms X}) = H^0(\ms X, \Omega_{\ms X}(R))$ and $H^1(\ms X, \mc O_{\ms X}) \cong H^1(\ms X, \mc O_{\ms X}(-R))$. Hence,
	by the exact sequence \cite[(6.2)]{Garnek_p_gp_covers}, the de~Rham cohomology of $\ms X$ might be computed
	as the hypercohomology of the complex $(\mc O_{\ms X}(-R) \stackrel{d}{\longrightarrow} \Omega_{\ms X}(R))$.
	We obtain:
	\begin{align*}
		H^1_{dR}(\ms X) &\cong \frac{\{ (\omega, \nu) \in H^0(V, \Omega_{\ms X}) \times k(\ms X) : \omega - d\nu \in \Omega_{\ms X}(R) \}}{\{ (df, f + \nu) : f \in \mc O_{\ms X}(V),
			\nu \in \mc O_{\ms X}(-R)_{\infty} \}}\\[0.5em]
		&\cong \frac{\{ (\omega, \nu) \in H^0(V, \Omega_{\ms X}) \times \frac{k(\ms X)}{\mc O_{\ms X}(-R)_{\infty}} : \omega - d\nu \in \Omega_{\ms X}(R) \}}{\{ (df, f) : f \in \mc O_{\ms X}(V)\}}\\[0.5em]
		&\cong \frac{\{ (\omega, \nu) \in H^0(V, \Omega_{\ms X}) \times \frac{\ms L}{\mf m_{\ms B}} : \omega - d\nu \in \Omega_{\ms B}^{\log} \}}{\{ (df, f) : f \in \mc O_{\ms X}(V) \}}
	\end{align*}
	(we used the isomorphism $k(\ms X)/\mc O_{\ms X}(-R)_{\infty} \cong \ms L/\mf m_B$ for the last equality).
	By Proposition~\ref{prop:key_fact_omega_L} $H^0(V, \Omega_{\ms X})$ may be identified with $\frac{\Omega_{\ms L}}{S_G(z)}$. Thus
	we may identify:
	\begin{align*}
		H^1_{dR}(z) \cong \left\{ (\omega, \nu) \in H^0(V, \Omega_{\ms X}) \times \frac{S_G^{\vee}(z)}{\mf m_{\ms B}}:
		\quad \omega - d\nu \in \Omega_{\ms B}^{\log} \right\}.
	\end{align*}
	The inclusion:
	\[
		H^0(V, \Omega_{\ms X}) \times S_G^{\vee}(z) \hookrightarrow H^0(V, \Omega_{\ms X}) \times \ms L
	\]
	induces a homomorphism $H^1_{dR}(z) \to H^1_{dR}(\ms X)$. Thus we obtain a commutative diagram:
	\begin{center}
		\begin{tikzcd}
			0 \arrow[r] & H^0(z) \arrow[r] \arrow[d] & {H^1_{dR}(z)} \arrow[r] \arrow[d] & H^1(z) \arrow[r] \arrow[d] & 0 \\
			0 \arrow[r] & {H^0(\ms X, \Omega_{\ms X})} \arrow[r]     & H^1_{dR}(\ms X) \arrow[r]                  & {H^1(\ms X, \mc O_{\ms X})} \arrow[r]      & 0,
		\end{tikzcd}
	\end{center}
	in which rows are exact and the left and right arrows are isomorphisms. Thus the middle arrow is also an isomorphism. This ends the proof.
\end{proof}

\section{$p$-group covers} \label{sec:p_gp}
In this section we assume that $k$, $G$ and $\pi : X \to Y$ are as in Theorem~\ref{thm:cohomology_of_G_covers}.
Recall that $B \subset Y(k)$ denotes the set of branch points of $\pi$ and let $R := \pi^{-1}(B)$
be the ramification locus. Also, by abuse of notation, for $Q \in Y(k)$ we write $G_{Q, i} := G_{P, i}$, $e_Q := e_{X/Y, P}$, $d_Q := d_{X/Y, P}$ etc. for any $P \in \pi^{-1}(Q)$. Note that these quantities don't depend on the choice
of $P$ by the condition~\ref{enum:A}. \\

\noindent Denote for any $Q \in B$:
\begin{align*}
	H^0_Q &:= \frac{\Omega_X(B)_Q}{\bigoplus_{g \in G} g^*(z) \Omega_Y(B)_Q},\\
	H^1_Q &:= \frac{\bigoplus_{g \in G} g^*(z^{\vee}) \mc O_Y(-B)_Q}{\mc O_X(-R)_Q},\\
	H^1_{dR, Q} &:= \left\{ (\omega, \nu) \in \frac{\Omega_{k(X)}}{\bigoplus_{g \in G} g^*(z) \Omega_Y(B)_Q} \times \frac{\bigoplus_{g \in G} g^*(z^{\vee}) \mc O_Y(-B)_Q}{\mc O_X(-R)_Q}:
	\quad \omega - d\nu \in \Omega_X(B)_Q \right\}.
\end{align*}
\begin{proof}[Proof of Theorem~\ref{thm:cohomology_of_G_covers}]
	Recall that by \cite[Theorem~1.1]{Garnek_p_gp_covers} we have the following isomorphisms of $k[G]$-modules:
	\begin{align*}
		H^0(X, \Omega_X) &\cong k[G]^{\oplus g_Y} \oplus I_{X/Y} \oplus \bigoplus_{Q \in B} H^0_Q,\\
		H^1(X, \mc O_X) &\cong k[G]^{\oplus g_Y} \oplus I_{X/Y}^{\vee} \oplus \bigoplus_{Q \in B} H^1_Q,\\
		H^1_{dR}(X) &\cong k[G]^{\oplus 2 g_Y} \oplus I_{X/Y} \oplus I_{X/Y}^{\vee} \oplus \bigoplus_{Q \in B} H^1_{dR, Q}.
	\end{align*}
	Fix $Q \in B$ and let $x \in k(Y)$ be a uniformizer at $Q$. We prove now that $H^0_Q \cong H^0(\ms X_Q, \Omega_{\ms X_Q})$.
	Let $\ms A := \wh{\mc O}_{Y, Q}$ and $\ms B := \wh{\mc O}_{X, Q}$. From now on, we use the notation of Section~\ref{sec:hkg}. In particular,
	denote
	by $\ms K$ and $\ms L$ the
	total fraction fields of $\ms A$ and~$\ms B$. Note that $\ms K = k((x))$ and $\ms L$ is a $G$-Galois \'{e}tale algebra over $\ms K$. The inclusion $\Omega_{k(X)} \hookrightarrow \Omega_{\ms L}$ induces a natural morphism:
	\begin{equation} \label{eqn:map_H0Q_H0Q_completed}
		H^0_Q := \frac{\Omega_{X}(R)_Q}{\bigoplus_{g \in G} g^*(z) \Omega_Y(B)_Q} \to 
		\frac{\Omega_{\ms B}^{\log}}{\bigoplus_{g \in G} g^*(z) \Omega_{\ms A}^{\log}} \cong H^0(\ms X_Q, \Omega_{\ms X_Q})
	\end{equation}
	(we used Theorem~\ref{thm:cohomology_of_HKG_etale_algebras} for the last isomorphism). The map~\eqref{eqn:map_H0Q_H0Q_completed} is injective, since
	if
	\[
		\omega = \sum_{g \in G} g^*(z) \omega_g \in \Omega_{X}(R)_Q \cap \left(\bigoplus_{g \in G} g^*(z) \Omega_{\ms A}^{\log} \right)
	\]
	then $\omega_g = \tr_{X/Y}(g^*(z^{\vee}) \cdot \omega) \in \Omega_{k(Y)} \cap \Omega_{\ms A}^{\log} = \Omega_Y(B)_Q$
	for every $g \in G$ and thus $\omega \in \bigoplus_{g \in G} g^*(z) \Omega_Y(B)_Q$.
	We prove now that the map~\eqref{eqn:map_H0Q_H0Q_completed} is surjective.
	Note that for any $\omega \in \Omega_{\ms L}$ there exists $\omega' \in \Omega_{k(X)}$ with $\omega - \omega' \in \bigoplus_{g \in G} g^*(z) \Omega_{\ms A}^{\log}$. Indeed, if $\omega = \sum_{g \in G} g^*(z) \omega_g$ then
	it suffices to take $\omega' := \sum_{g \in G} g^*(z) \omega_g^{< 0}$, where
	for $\eta = \sum_{i \in \ZZ} a_i x^i \, dx \in \Omega_{\ms K} = k((x)) \, dx$,
	we put $\eta^{< 0} := \sum_{i < 0} a_i x^i \, dx$. This proves that~\eqref{eqn:map_H0Q_H0Q_completed} is an isomorphism.\\
	Similarly, the inclusion $k(X) \hookrightarrow \ms L$ induces an isomorphism:
	\begin{equation*}
		H^1_Q := \frac{\bigoplus_{g \in G} g^*(z^{\vee}) \mc O_Y(-B)_Q}{\mc O_X(-R)_Q} \to
		\frac{\bigoplus_{g \in G} g^*(z^{\vee}) \mf m_{\ms B}}{\mf m_{\ms A}} \cong H^1(\ms X_Q, \mc O_{\ms X_Q}).
	\end{equation*}
	For the de~Rham cohomology, by Theorem~\ref{thm:cohomology_of_HKG_etale_algebras} the inclusion $\Omega_{k(X)} \times k(X) \hookrightarrow \Omega_{\ms L} \times \ms L$
	induces a map $H^1_{dR, Q} \to H^1_{dR}(\ms X_Q)$, which fits in the commutative diagram with exact rows:
	\begin{center}
		\begin{tikzcd}
			0 \arrow[r] & H^0_{Q} \arrow[r] \arrow[d] & {H^1_{dR, Q}} \arrow[r] \arrow[d] & H^1_{Q} \arrow[r] \arrow[d] & 0 \\
			0 \arrow[r] & {H^0(\ms X_Q, \Omega_{\ms X_Q})} \arrow[r]     & H^1_{dR}(\ms X_Q) \arrow[r]                  & {H^1(\ms X_Q, \mc O_{\ms X_Q})} \arrow[r]      & 0.
		\end{tikzcd}
	\end{center}
	Since the outer two arrows are isomorphisms, we deduce that the map $H^1_{dR, Q} \to H^1_{dR}(\ms X_Q)$ is also an isomorphism.\\
	
	\noindent Finally, note that by~\eqref{eqn:HKG_is_Ind_HKG} $\ms X_Q = \Ind^G_{G_Q} \ms X_Q^{\circ}$, which easily leads to the isomorphisms:
	\begin{align*}
		H^0(\ms X_Q, \Omega_{\ms X_Q}) &\cong \Ind^G_{G_Q} H^0(\ms X_Q^{\circ}, \Omega_{\ms X_Q^{\circ}})\\
		H^1(\ms X_Q, \mc O_{\ms X_Q}) &\cong \Ind^G_{G_Q} H^1(\ms X_Q^{\circ}, \mc O_{\ms X_Q^{\circ}})\\
		H^1_{dR}(\ms X_Q) &\cong \Ind^G_{G_Q} H^1_{dR}(\ms X_Q^{\circ}).
	\end{align*}
	This ends the proof.
\end{proof}
For later use, we prove a lemma that simplifies computation of $I_{X/Y}$.
In the sequel we identify the relative augmentation ideal $I_{G, H}$ (as defined in Section~1) with:
\[
I_{G, H} = \left\{ \sum_{g \in G} a_g g \in k[G] : \sum_{g \in g_0 H} a_g = 0 \quad \forall {g_0 \in G} \right\}.
\]
For any $Q \in B$ and $g \in G$, let $g_Q \in \bigoplus_{B} k[G]$ be the element with $g$ on the $Q$-th component and $0$ on other components.
\begin{Lemma} \label{lem:IXY_and_IXY_prim}
	Suppose that $S \subset B$ is such that for every $Q \in B$
	there exists $Q' \in S$ such that $G_Q \subset G_{Q'}$. Then:
	\begin{align*}
		I_{X/Y} \cong \bigoplus_{Q \in B \setminus S} I_{G, G_Q}
		\oplus I(S),
	\end{align*}
	where $I(S) := \ker(\Sigma : \bigoplus_{Q \in S} I_{G, G_{Q}} \to I_G)$.
\end{Lemma}
\begin{proof}
	We prove it by induction on $|B \setminus S|$. For $B = S$ this is straightforward.
	To prove the induction step, it suffices to show that if $S = S' \cup \{ Q_0 \}$
	and $G_{Q_0} \subset G_{Q_1}$ for some $Q_1 \in S'$,
	then:
	\[
		I(S) \cong I_{G, G_{Q_0}} \oplus I(S').
	\]
	We treat both sides as submodules of $\bigoplus_B k[G]$.
	The isomorphisms are given by:
	\begin{align*}
		\Phi : I(S) &\to I_{G, G_{Q_0}} \oplus I(S'),\\
		\sum_{Q, g} a_{Qg} \cdot g_Q &\mapsto \sum_{g} a_{Q_0 g} \cdot g_{Q_0}
		+ \sum_{g} a_{Q_0 g} \cdot g_{Q_1}
		+ \sum_{Q \in S'} \sum_{g} a_{Qg} \cdot g_Q,\\
		\Psi : I_{G, G_{Q_0}} \oplus I(S') &\to I(S),\\
		\sum_{Q, g} a_{gQ} \cdot g_Q &\mapsto \sum_{g} a_{Q_0 g} \cdot g_{Q_0}
		- \sum_{g} a_{Q_0 g} \cdot g_{Q_1}
		+ \sum_{Q \in S'} \sum_{g} a_{Qg} \cdot g_Q
	\end{align*}
	(we abbreviate $\sum_{Q \in S}$ to $\sum_Q$ and $\sum_{g \in G}$ to $\sum_g$).
	Note that $\Phi$ is well-defined. Indeed, for any $\sum_{Q, g} a_{Qg} \cdot g_Q \in I(S)$:
	\begin{itemize}
		\item $\sum_{g} a_{Q_0 g} \cdot g \in I_{G, G_{Q_0}} \subset I_{G, G_{Q_1}}$,
		
		\item $\sum_{g} a_{Q_0 g}\cdot g_{Q_1}
		+ \sum_{Q \in S'} a_{Qg}\cdot g_Q \in \ker(\bigoplus_{S'} k[G] \to k[G])$, since
		\[
			\sum_{g} a_{Q_0 g}\cdot g + \sum_g \sum_{Q \in S'} a_{Qg}\cdot g = \sum_{g} \left(\sum_{Q \in S} a_{Q g} \right) \cdot g.
		\]
		and $\sum_{Q, g} a_{Qg} g_Q \in \ker(\bigoplus_{S} k[G] \to k[G])$.
	\end{itemize}
	Analogously one checks that $\Psi$ is well-defined. It is easy to check that $\Phi$ and
	$\Psi$ are equivariant and mutually inverse.
\end{proof}

\section{Klein four covers} \label{sec:klein_covers}
In this section we recall basic facts concerning $\VV_4$-covers in characteristic~$2$. Keep
previous notation with $p = 2$ and $G = \VV_4 = \{ e, \sigma, \tau, \sigma \tau \}$. Write $H_0 = \langle \tau \rangle$,
$H_1 = \langle \sigma \rangle$, $H_{\infty} = \langle \sigma \tau \rangle$. Let also for any subgroup $H \le \VV_4$:
\begin{align*}
	B(H) &:= \{ Q \in B : G_Q = H \},\\
	B'(H) &:= \{ Q \in B : G_{Q, i} = H, \quad G_{Q, i+1} = \{ e \}  \, \textrm{ for some } i \}.
\end{align*}
We discuss now how to find $B(\VV_4)$, $B(H_i)$, $B'(H_i)$ and $d_Q$, given the equations defining~$X$.
Recall that the function field of $X$ is given by equations of the form:
\begin{equation} \label{eqn:klein_eqn}
	y_0^2 + y_0 = h_0, \quad y_1^2 + y_1 = h_1,%
\end{equation}
where $h_0, h_1 \in k(Y)$, $\sigma(y_0) = y_0 + 1$, $\sigma(y_1) = y_1$ and
$\tau(y_0) = y_0$, $\tau(y_1) = y_1 + 1$. Denote $y_{\infty} := y_0 + y_1$, $h_{\infty} := h_0 + h_1$.\\

We say that the equations~\eqref{eqn:klein_eqn} are in \emph{standard form} at a point
$Q \in Y(k)$, if the function $h_i$ is either regular at $Q$ or has a pole of odd order
for every $i = 0, 1, \infty$. Note that any equation
can be brought to the standard form at a given point by successively subtracting powers
of a uniformizer. Moreover, if $Y = \PP^1$, one can find an equation of the form~\eqref{eqn:klein_eqn},
which is in standard form at any $Q \in Y(k)$.\\

The article \cite{Bleher_Camacho_Holomorphic_differentials} defined
two local invariants $m_Q$ and $M_Q$ of the cover $\pi$ associated to every $Q \in Y(k)$ as follows.
The number $m_Q$ is the minimum of the lower ramification jumps of the covers $X/H_i \to Y$ for $i \in \{ 0, 1, \infty \}$ at $Q$. Similarly, $M_Q$ is the maximum 
of those jumps. Note that for any $i \in \{ 0, 1, \infty\}$ the cover $X/H_i \to Y$ is given by the equation $y_i^2 + y_i = h_i$.
Its lower ramification jump at $Q$ equals $\pole_Q(h_i)$ (cf. \cite[Lemma~4.2]{Garnek_equivariant}), where $\pole_Q(h) := \max \{ 0, -\ord_Q(h) \}$ for any $h \in k(Y)$.
Therefore:
\begin{align*}
	m_Q &:= \min \{ \pole_Q(h_0), \pole_Q(h_1), \pole_Q(h_{\infty}) \}, \\
	M_Q &:= \max \{ \pole_Q(h_0), \pole_Q(h_1), \pole_Q(h_{\infty}) \}.
\end{align*}
In order to simplify the formulation of Theorem~\ref{thm:de_rham_of_klein}, we modify these definitions slightly.
Namely, we replace $m_Q$ by $\max\{ 1, m_Q \}$ and $M_Q$ by $\max\{ 1, M_Q \}$. In other words, if $G_Q \neq \VV_4$, then we put
$m_Q = 1$ (instead of $0$) and if $G_Q = \{ e \}$, then also $M_Q = 1$.\\

The discussion above easily allows to determine the sets $B(H)$, $B'(H)$ for any $H \le G$:
\begin{itemize}
	\item We have $Q \in B(\VV_4)$ if and only if $h_0$, $h_1$, $h_{\infty}$ have poles at $Q$.
	In this case, $Q \in B'(H_a)$ if and only if $m_Q \neq M_Q$ and $m_Q = - \ord_Q(h_a)$. Moreover,
	$Q \in B'(\VV_4)$ if and only if $m_Q = M_Q$, i.e. $\ord_Q(h_0) = \ord_Q(h_1) = \ord_Q(h_{\infty}) < 0$.
	
	\item We have $Q \in B(H_a)$ if and only if $h_a \in \mc O_{Y, Q}$  and $h_b$, $h_c$ have a pole at~$Q$
	for $\{ a, b, c \} = \{ 0, 1, \infty \}$.
	
	\item We have $Q \not \in B$ if and only if $h_0$, $h_1$, $h_{\infty} \in \mc O_{Y, Q}$.
\end{itemize}
By \cite[Section~3]{Wu_Scheidler_Ramification_groups_AS_extensions} for any $Q \in B(\VV_4)$:
\begin{align*}
	\VV_4 = G_{Q, 0} = \ldots = G_{Q, m_Q} > G_{Q, m_Q + 1} = \ldots = G_{Q, m_Q + 2(M_Q - m_Q)}
	> G_{Q, m_Q + 2(M_Q - m_Q) + 1} = \{ e \}
\end{align*}
and for $Q \in B(H_a)$:
\begin{align*}
	\VV_4 > H_a = G_{Q, 0} = \ldots = G_{Q, M_Q} > G_{Q, M_Q + 1} = \{ e \}.
\end{align*}
In particular:
\begin{align*}
	d_Q = 
	\begin{cases}
		m_Q + 2M_Q + 3, & G_Q = \VV_4,\\
		M_Q + 1, & \# G_Q = 2,\\
		0, & Q \not \in B.
	\end{cases}
\end{align*}
Moreover, if $Q \in B(\VV_4) \cap B'(H_a)$, $\pi^{-1}(Q) = \{ P \}$ then $\ord_P(h_a) = -4m_Q$, $\ord_P(y_a) = -2m_Q$ and $\ord_P(h_b) = -4 M_Q$, $\ord_P(y_b) = -2M_Q$ for $b \neq a$.
\begin{Proposition} \label{prop:klein_of_P1_has_magical_elt}
	Keep the above notation and suppose that $Y = \PP^1$. Then there exists $z \in k(X)$ satisfying~\ref{enum:B}.
\end{Proposition}
\begin{proof}
	Suppose that $X$ is given by the equations~\eqref{eqn:klein_eqn}, which are in standard form at every $Q \in \PP^1(k)$.
	Then for $P \in R$ and $i \in \{0, 1\}$:
	\begin{equation} \label{eqn:ord_y_i_at_least_d'}
		\ord_P(y_i) \ge -d_P'.
	\end{equation}
	Indeed, if $G_P = \VV_4$, then $\ord_P(y_i) \ge -2 M_Q > -(2M_Q+m_Q) = -d_P'$.
	If $G_P = H_a$, then $\ord_P(y_i) \ge -M_Q = -d_P'$.
	
	For any $Q \in B$, let $z_Q \in k(X)$ be an element such that $\tr_{X/Y}(z_Q) = 1$ and $\ord_P(z_Q) \ge -d_P'$ for every $P \in \pi^{-1}(Q)$.
	Note that such an element $z_Q$ exists by Corollary~\ref{cor:existence_of_magical}. The elements $1, y_0, y_1, y_0 y_1$ form the $k(Y)$-linear basis of $k(X)$
	and $\tr_{X/Y}(1) = \tr_{X/Y}(y_0) = \tr_{X/Y}(y_1) = 0$, $\tr_{X/Y}(y_0 y_1) = 1$. Therefore $z_Q$ must be of the form:
	\[
		z_Q = a_Q + b_Q y_0 + c_Q y_1 + y_0 y_1
	\]
	for some $a_Q, b_Q, c_Q \in k(\PP^1)$. Let $a, b, c \in k(\PP^1)$ be functions that are regular outside of~$B$ and such that
	\begin{equation} \label{eqn:a-aQ_etc}
		a - a_Q, \, b - b_Q, \, c - c_Q \in \mc O_{\PP^1, Q} \qquad \textrm{ for every } Q \in B.
	\end{equation}
	We will show that
	\[
		z := a + b y_0 + c y_1 + y_0 y_1 \in k(X)
	\]
	satisfies~\ref{enum:B}. Indeed, $\tr_{X/Y}(z) = 1$. Moreover, $z$ is regular outside of $B$ and for every $Q \in B$:
	\[
		z = z_Q + (z - z_Q) = z_Q + (a - a_Q) + (b - b_Q) y_0 + (c - c_Q) y_1
	\]
	has valuation at least $-d_Q'$ by~\eqref{eqn:ord_y_i_at_least_d'}, ~\eqref{eqn:a-aQ_etc} and by the definition of $z_Q$. This ends the proof.
\end{proof}
\section{The de~Rham cohomology of $\VV_4$-covers} \label{sec:de_Rham_of_Klein}
In this section we prove a more precise version of Theorem~\ref{thm:de_rham_of_klein_intro}.
Keep notation from the previous section. Suppose that $\pi$ satisfies condition~\ref{enum:B}. Since
$\langle G_Q : Q \in B \rangle = \VV_4$ by \cite[Lemma~3.6]{Garnek_p_gp_covers}, we may distinguish three cases:
\begin{itemize}
	\item \bb{Case 1:} $B(\VV_4) \neq \varnothing$,
	
	\item \bb{Case 2:} $B(\VV_4) = \varnothing$ and $B(H_0), B(H_1), B(H_{\infty}) \neq \varnothing$,
	
	\item \bb{Case 3:} $B(\VV_4) = B(H_a) = \varnothing$ and $B(H_b) = B(H_c) \neq \varnothing$,
	\item[] where $\{ a, b, c \} = \{ 0, 1, \infty \}$.
\end{itemize}
We adopt the convention $H_{-1} := \VV_4$, $N_{2, -1} := 0$.
\begin{Theorem} \label{thm:de_rham_of_klein}
	Let $k$ be an algebraically closed field of characteristic~$2$.
	Suppose that $\pi : X \to Y$ is a $\VV_4$-cover of smooth projective
	curves over $k$ that satisfies the condition~\ref{enum:B} (this holds automatically for example
	if $Y = \PP^1$). Keep the above notation. Then:
	\begin{align*}
		H^1_{dR}(X) &\cong k[\VV_4]^{\oplus 2 g_Y} \oplus I_{X/Y} \oplus I_{X/Y}^{\vee}\\
		&\oplus \bigoplus_{i \in \{ -1, 0, 1, \infty \}} \, \bigoplus_{Q \in B'(H_i)} N_{2, i}^{M_Q - m_Q} \oplus M_{3, 1}^{(m_Q-1)/2} \oplus M_{3, 2}^{(m_Q-1)/2},
	\end{align*}
	where:
	\[
	I_{X/Y} \cong
	\begin{cases}
		M_{3, 1}^{\# B(\VV_4) - 1} \oplus N_{2, 0}^{\# B(H_0)} \oplus N_{2, 1}^{\# B(H_1)} \oplus N_{2, \infty}^{\# B(H_{\infty})}, &
		\textrm{ in Case 1,}\\[10pt]
		M_{3, 2} \oplus N_{2, 0}^{\# B(H_0) - 1} \oplus N_{2, 1}^{\# B(H_1) - 1} \oplus N_{2, \infty}^{\# B(H_{\infty}) - 1}, &
		\textrm{ in Case 2,}\\[10pt]
		k \oplus N_{2, b}^{\# B(H_b) - 1} \oplus N_{2, c}^{\# B(H_c) - 1}, & \textrm{ in Case 3}.
	\end{cases}
	\]
\end{Theorem}
Theorem~\ref{thm:de_rham_of_klein} will be proven at the end of this section. Note that Theorem~\ref{thm:de_rham_of_klein_intro} is a direct corollary of
Theorem~\ref{thm:de_rham_of_klein}. We discuss now numerical examples.
\begin{Example} \label{ex:elliptic_curve_ctn}
	Let $\pi : X \to Y$ be as in Example~\ref{ex:HKG_covers}. Thus $Y$ is the elliptic curve
	with the affine equation $w^2 + w = u^3$ and $X$ is given by~\eqref{eqn:klein_eqn} with $h_0 = w^3 + \frac 1{w^7}, h_1 = w^5 + \frac 1{w^7}$. Then $\pi$ is unramified outside of $\{ Q_1, Q_2 \}$,
	since those are the only poles of $h_0$ and $h_1$, and is in standard form
	at $Q_1$ and at $Q_2$. Hence, since $\ord_{Q_1}(w) = 3$, $\ord_{Q_2}(w) = -3$ and $h_{\infty} \in \mc O_{Y, Q_1}$,
	we have:
	\begin{align*}
		(m_{Q_1}, M_{Q_1}) &= (1, \pole_{Q_1}(h_1))
		= (1, 21),\\
		(m_{Q_2}, M_{Q_2}) &= (\pole_{Q_2}(h_0), \pole_{Q_2}(h_1))
		= (9, 15).
	\end{align*}
	By the discussion in Section~\ref{sec:klein_covers}, $Q_1 \in  B(H_{\infty})$, $Q_2 \in B(\VV_4) \cap B'(H_0)$ (hence we are in Case~1).
	Moreover, we see that
	$d_{X/Y, Q_1} = 22$, $d_{X/Y, Q_2} = 42$. Hence, by Riemann--Hurwitz formula:
	\[
	2(g_X - 1) = 4 \cdot 2 \cdot (1 - 1) + 2 \cdot 22 + 42,
	\]
	i.e. $g_X = 44$. Let $X_1$ be the $\ZZ/2$-cover
	of $Y$ given by $y_{\infty}^2 + y_{\infty} = h_{\infty}$.  
	Then $d_{X_1/Y, Q_1} = 16$ (cf. \cite[Lemma~4.2]{Garnek_equivariant}) and $g_{X_1} = 9$. Hence
	$g_{X_1} > 2 \cdot 2 \cdot g_Y$ and $g_X > 2 \cdot 2 \cdot g_{X_1}$. 
	By \cite[Lemma~7.3, Lemma~7.2]{Garnek_p_gp_covers} both covers $X \to X_1$ and $X_1 \to Y$
	have magical elements. Hence by \cite[Lemma 8.1]{Garnek_p_gp_covers}, the cover $X \to Y$
	also has a magical element. By Theorem~\ref{thm:de_rham_of_klein}:
	\[
		H^1_{dR}(X) \cong k[\VV_4]^{\oplus 2} \oplus N_{2, \infty}^{\oplus 22} \oplus N_{2, 0}^{\oplus 6} \oplus M_{3, 1}^{\oplus 4}
		\oplus M_{3, 2}^{\oplus 4}.
	\]
\end{Example}
\begin{Example} \label{ex:cover_of_P1}
	Let $\pi : X \to Y$ be given by the equations~\eqref{eqn:klein_eqn}, where $Y = \PP^1$ and $h_0, h_1 \in k(Y)$,
	$h_0 = x^3 + x + \frac{1}{(x - 1)^7}$, $h_1 = x^3 + x + \frac{1}{x^5}$.
	Let $Q_i := i \in \PP^1$ for $i = 0, 1, \infty$. Then the branch locus of $\pi$ is contained in $\{ Q_0, Q_1, Q_{\infty} \}$,
	as those are the only poles of $h_0$ and~$h_1$. Moreover:
	\begin{align*}
		(\pole_{Q_0}(h_0), \pole_{Q_0}(h_1), \pole_{Q_0}(h_{\infty})) &= (0, 5, 5),\\
		(\pole_{Q_1}(h_0), \pole_{Q_1}(h_1), \pole_{Q_1}(h_{\infty})) &= (7, 0, 7),\\
		(\pole_{Q_{\infty}}(h_0), \pole_{Q_{\infty}}(h_1), \pole_{Q_{\infty}}(h_{\infty})) &= (3, 3, 0).
	\end{align*}
	By the discussion in Section~\ref{sec:klein_covers}, $Q_i \in B(H_i)$ for $i = 0, 1, \infty$ and
	$(M_{Q_0}, M_{Q_1}, M_{Q_{\infty}}) = (5, 7, 3)$. Thus we are in Case 2 and by Theorem~\ref{thm:de_rham_of_klein}:
	\[
	H^1_{dR}(X) \cong M_{3, 1} \oplus M_{3, 2} \oplus N_{2, 0}^{\oplus 4} \oplus N_{2, 1}^{\oplus 6} \oplus N_{2, \infty}^{\oplus 2}.
	\]
\end{Example}

We start the proof of Theorem~\ref{thm:de_rham_of_klein} by computing the module $I_{X/Y}$.
\begin{Lemma} \label{lem:IXY_of_Klein_covers}
	Keep the above notation. Then:
	\[
	I_{X/Y} \cong
	\begin{cases}
		M_{3, 1}^{\# B(\VV_4) - 1} \oplus N_{2, 0}^{\# B(H_0)} \oplus N_{2, 1}^{\# B(H_1)} \oplus N_{2, \infty}^{\# B(H_{\infty})}, &
		\textrm{ in Case 1,}\\[10pt]
		M_{3, 2} \oplus N_{2, 0}^{\# B(H_0) - 1} \oplus N_{2, 1}^{\# B(H_1) - 1} \oplus N_{2, \infty}^{\# B(H_{\infty}) - 1}, &
		\textrm{ in Case 2,}\\[10pt]
		k \oplus N_{2, b}^{\# B(H_b) - 1} \oplus N_{2, c}^{\# B(H_c) - 1}, & \textrm{ in Case 3}.
	\end{cases}
	\]
\end{Lemma}
\begin{proof}
	Before the proof note that $I_{\VV_4} \cong M_{3,1}$,
	which follows by computing the matrices of $\sigma$
	and $\tau$ in the basis $(e + \sigma + \tau + \sigma \tau, e + \sigma, e + \tau)$ and comparing them with the matrices in Table~\ref{tab:modules}. Moreover, $I_{\VV_4, H_i} \cong N_{2, i}$ for $i \in \{ 0, 1, \infty \}$. Indeed, for example the matrices of $\sigma$
	and $\tau$ acting on $I_{\VV_4, H_0}$ in the basis $(e + \sigma + \tau + \sigma \tau, e + \tau)$ match the matrices in Table~\ref{tab:modules}.\\ \mbox{}\\
	\bb{Case 1:} Let $Q' \in B(\VV_4)$. Then, by Lemma~\ref{lem:IXY_and_IXY_prim}:
	\begin{align*}
		I_{X/Y} &\cong \bigoplus_{Q \neq Q'} I_{\VV_4, G_Q} \oplus I(\{ Q' \})\\
		&\cong M_{3, 1}^{\# B(\VV_4) - 1} \oplus N_{2, 0}^{\# B(H_0)} \oplus N_{2, 1}^{\# B(H_1)} \oplus N_{2, \infty}^{\# B(H_{\infty})} \oplus \ker(I_{\VV_4} \to I_{\VV_4})\\
		&\cong M_{3, 1}^{\# B(\VV_4) - 1} \oplus N_{2, 0}^{\# B(H_0)} \oplus N_{2, 1}^{\# B(H_1)} \oplus N_{2, \infty}^{\# B(H_{\infty})}.
	\end{align*}
	\bb{Case 2:} Suppose that $Q_i \in B(H_i)$ for $i \in \{ 0, 1, \infty \}$. Then, by Lemma~\ref{lem:IXY_and_IXY_prim}
	for $S := \{ Q_0, Q_1, Q_{\infty} \}$:
	\begin{align*}
		I_{X/Y} &\cong \bigoplus_{Q \in B \setminus S} I_{\VV_4, G_Q} \oplus I(S)\\
		&\cong N_{2, 0}^{\# B(H_0) - 1} \oplus N_{2, 1}^{\# B(H_1) - 1}
		\oplus N_{2, \infty}^{\# B(H_{\infty}) - 1} \oplus I(S).
	\end{align*}		
	Moreover, one easily checks that the basis of $I(S)$ is:
	\begin{align*}
		e_1 &:= e_{Q_0} + \sigma_{Q_0} + \tau_{Q_0} + (\sigma \tau)_{Q_0} + e_{Q_{\infty}} + \sigma_{Q_{\infty}} + \tau_{Q_{\infty}}
		+ (\sigma \tau)_{Q_{\infty}},\\
		e_2 &:= e_{Q_1} + \sigma_{Q_1} + \tau_{Q_1} + (\sigma \tau)_{Q_1} + e_{Q_{\infty}} + \sigma_{Q_{\infty}} + \tau_{Q_{\infty}}
		+ (\sigma \tau)_{Q_{\infty}},\\
		e_3 &:= \sigma_{Q_0} + (\sigma \tau)_{Q_0} + \tau_{Q_1} + (\sigma \tau)_{Q_1} + \sigma_{Q_{\infty}} + \tau_{Q_{\infty}}.
	\end{align*}
	Then $\sigma(e_1) = \tau(e_1) = e_1$, $\sigma(e_2) = \tau(e_2) = e_2$
	and $\sigma(e_3) = e_3 + e_1$, $\tau(e_3) = e_3 + e_2$. This shows that
	$\Span_{k}(e_1, e_2, e_3) \cong M_{3, 2}$ (cf. Table~\ref{tab:modules}).\\ \mbox{} \\
	\bb{Case 3:} Suppose that $Q_b \in B(H_b)$, $Q_c \in B(H_c)$. Using Lemma~\ref{lem:IXY_and_IXY_prim} for
	$S := \{ Q_b, Q_c \}$ analogously as in (a) and (b):
	\begin{align*}
		I_{X/Y} \cong N_{2, b}^{\# B(H_b) - 1} \oplus N_{2, c}^{\# B(H_c) - 1} \oplus I(S).
	\end{align*}
	The $k[\VV_4]$-module $I(S)$ is one-dimensional, which implies it must be the trivial representation of $\VV_4$.
\end{proof}
The following result will be proven in Section~\ref{sec:klein_hkg_covers}.
\begin{Proposition} \label{prop:dR_of_Klein_HKG}
	Let $k$ be an algebraically closed field of characteristic $2$ and $G = \ZZ/2 \times \ZZ/2$.
	Suppose that $\ms X \to \PP^1$ is a $\VV_4$-HKG-cover and $B = B'(H_i)$.
	Denote $m := m_{\ms X/\PP^1, \infty}$, $M := M_{\ms X/\PP^1, \infty}$.
	Then, as $k[\VV_4]$-modules:
	\begin{equation*}
		H^1_{dR}(\ms X) \cong N_{2, i}^{\oplus (M - m)} \oplus M_{3, 1}^{\oplus (m - 1)/2} \oplus M_{3, 2}^{\oplus (m - 1)/2}.
	\end{equation*}
\end{Proposition}
We show now how Proposition~\ref{prop:dR_of_Klein_HKG} together with previous results
implies Theorem~\ref{thm:de_rham_of_klein}.
\begin{proof}[Proof of Theorem~{\ref{thm:de_rham_of_klein}}]
	Recall that if $Y = \PP^1$, the cover $\pi : X \to Y$ has a magical element by 
	Proposition~\ref{prop:klein_of_P1_has_magical_elt}.
	Using Theorem~\ref{thm:cohomology_of_G_covers}:
	\begin{align*}
		H^1_{dR}(X) &\cong k[\VV_4]^{\oplus 2 g_Y} \oplus I_{X/Y} \oplus I_{X/Y}^{\vee} \oplus \bigoplus_{i \in \{ -1, 0, 1, \infty \}} \bigoplus_{Q \in B'(H_i)} H^1_{dR}(\ms X_Q),
	\end{align*}
	The module $I_{X/Y}$ was computed in Lemma~\ref{lem:IXY_of_Klein_covers}.
	Suppose now that $Q \in B'(H_i)$. Then, by Proposition~\ref{prop:dR_of_Klein_HKG}:
	\begin{equation*}
		H^1_{dR}(\ms X_Q) \cong N_{2, i}^{\oplus (M_Q - m_Q)} \oplus M_{3, 1}^{\oplus (m_Q-1)/2} \oplus M_{3, 2}^{\oplus (m_Q-1)/2}.
	\end{equation*}
	The proof follows.
\end{proof}

\section{HKG-covers for $\VV_4$} \label{sec:klein_hkg_covers}
This section will be devoted to the proof of Proposition~\ref{prop:dR_of_Klein_HKG}.
We keep the notation of Proposition~\ref{prop:dR_of_Klein_HKG}. 
Suppose at first that $G_{\infty} = H_i$. In particular, $m = 1$. Then $H^1_{dR}(\ms X^{\circ}) \cong k^{\oplus (M - 1)}$ by \cite[Corollary~1.2]{Garnek_p_gp_covers}. Hence:
\[
H^1_{dR}(\ms X) \cong \Ind^{\VV_4}_{H_i} H^1_{dR}(\ms X^{\circ}) \cong (\Ind^{\VV_4}_{H_i} k)^{\oplus (M - 1)} \cong N_{2, i}^{\oplus (M - 1)},
\]
which proves Proposition~\ref{prop:dR_of_Klein_HKG} in this case.\\

From now on, we assume that $G_{\infty} = \VV_4$.
Without loss of generality, we may assume that $\ord_{\infty}(h_0) = -m$ and $\ord_{\infty}(h_1) = -M$. Indeed,
this is immediate if $\infty \in B'(\VV_4)$, since then $\ord_{\infty}(h_0) = \ord_{\infty}(h_1) = \ord_{\infty}(h_{\infty}) = -m = -M$.
If $\infty \in B'(H_1)$ (resp. $\infty \in B'(H_{\infty})$), then 
we may twist the action of $\VV_4$ by swapping $\tau$ and $\sigma$
(respectively $\tau$ and $\sigma \tau$). Under this twist, the module $N_{2, 0}$ maps to $N_{2, 1}$
(respectively $N_{2, \infty}$).\\

We may assume that the function field of $\ms X$
is given by the equations~\eqref{eqn:klein_eqn} in standard local form at every $Q \in \PP^1$.
Thus $h_0, h_1 \in k[x]$ and $\deg(h_0) = m$, $\deg(h_1) = M$.
For brevity, we write $m' := \frac{m-1}2$, $M' := \frac{M-1}2$, $M'' := \frac{M - m}{2}$ and $M_0 := \left \lfloor \frac{2M - m - 1}{4} \right\rfloor$.
Recall that $g_{\ms X} = \frac 12 d_{\infty}'' = \frac 12 (m + 2M - 3)$.
Let $P \in \ms X(k)$ be the point lying over $\infty \in \PP^1(k)$ and $V := X \setminus \{ P \}$.

Recall that there exists $\alpha \in k[x]$ of degree $M''$ such that $\ord_P(y_1 + \alpha y_0) = -2M + m$.
Indeed, by \cite[Corollary~3.10 and the proof of Theorem~3.11]{Wu_Scheidler_Ramification_groups_AS_extensions} 
there exists $\alpha \in k(x)$ with the desired properties. By taking the polynomial part,
we may assume that $\alpha \in k[x]$. Denote the coefficients of $\alpha$ as follows:
\[
\alpha = \alpha_0 + \alpha_1 \cdot x + \ldots + \alpha_{M''} \cdot x^{M''}.
\]
In order to prove Proposition~\ref{prop:dR_of_Klein_HKG} we construct an explicit basis of $H^1_{dR}(\ms X)$. Denote by $h_0'$, $h_1'$ the derivatives
of the polynomials $h_0, h_1 \in k[x]$. Note that any rational function
$T \in k(x)$ can be uniquely written as a sum $T = [T]_{\ge 0} + [T]_{< 0}$,
where $[T]_{\ge 0} \in k[x]$ and $[T]_{< 0}$ is a rational function, whose
numerator has lower degree then denominator.\\

Consider the elements of $\Omega_{\ms X}(V) \times k(\ms X)$ given by the following formulas:
\begin{itemize}[leftmargin=*]
	\item for $i \in I_1 := \{ 0, \ldots, m' - 1 \}$:
	\begin{alignat*}{4}
		a_i &:= \left(x^i \, dx, 0 \right),\\
		b_i &:= \left(y_1 x^i \, dx, \frac{y_0 y_1 x^i \alpha}{h_1' + \alpha h_0'} \right),\\
		c_i &:= \left(y_0 x^i \, dx, \frac{y_0 y_1 x^i}{h_1' + \alpha h_0'} \right),
		\shortintertext{\item for $i \in I_2 := \{ 1, \ldots, m' \}$:}
		e_i &:= \left(\left[\frac{h_1'}{x^i} \right]_{\ge 0} \, dx, \frac{y_1}{x^i} \right),\\
		f_i &:= \left(\left[\frac{h_0'}{x^i} \right]_{\ge 0} \, dx, \frac{y_0}{x^i} \right),\\
		g_i &:= \left(y_1 \cdot \left[\frac{h_0'}{x^i} \right]_{\ge 0} \, dx + y_0 \left[\frac{h_1'}{x^i} \right]_{\ge 0} \, dx, \frac{y_0 y_1}{x^i} + \frac{i \cdot y_1 \alpha h_0}{h_1' \cdot x^{i + 1}} \right),
		\intertext{\item for $i \in I_3 := \{ m' + 1, \ldots, M - 1 - m' \}$:}
		u_i &:= \left( \left[ \frac{h_1'}{x^i} \right]_{\ge 0} \, dx, \frac{y_1}{x^i} \right),\\
		v_i &:= \left(y_0 \left[ \frac{h_1'}{x^i} \right]_{\ge 0} \, dx, \frac{y_0 y_1}{h_1' + \alpha h_0'} \left[ \frac{h_1'}{x^i} \right]_{\ge 0} \right).
	\end{alignat*}
\end{itemize}
In the sequel we will need to compute valuations of several functions and forms on $\ms X$.
To this end we use~\eqref{eqn:valuation_of_diff_form} as well as the following relations:
\begin{itemize}
	\item $\ord_P(dx) = -8 + d_P = 2M + m - 5$,
	\item $\ord_P(h_0) = -4m$, $\ord_P(h_1) = -4M$,
	\item $\ord_P(y_0) = -2m$, $\ord_P(y_1) = -2M$,
	\item $\ord_P(y_1 + \alpha y_0) = -2M + m$,
	\item $\ord_P(y_0 y_1 + \alpha h_0) = -2M - m$.
\end{itemize}
For the proof of the last relation, note that:
\begin{align*}
	\ord_P(y_0 y_1 + \alpha h_0) &= \ord_P(y_0 y_1 + \alpha (y_0^2 + y_0)) = \ord_P(y_0 \cdot ((y_1 + \alpha y_0) + \alpha))\\
	&= -2m -2M + m = -2M - m.
\end{align*}
In the sequel we use also the fact that for any $T \in k(x)$:
\begin{equation} \label{eqn:P<0_is_holo}
	[T(x)]_{<0} \, dx, \quad y_0 \cdot [T(x)]_{<0} \, dx, \quad y_1 \cdot [T(x)]_{<0} \, dx \in \Omega_{\ms X, P}.
\end{equation}
Indeed, note that:
\begin{align*}
	\ord_P([T(x)]_{<0} \, dx) &\ge \ord_P(y_0 \cdot [T(x)]_{<0} \, dx) \ge \ord_P(y_1 \cdot [T(x)]_{<0} \, dx)\\
	&\ge -2M + 4 + 2M + m - 5 = m - 1 \ge 0.
\end{align*}
\begin{Lemma}
	The elements $(a_i)_{i \in I_1}$, $(b_i)_{i \in I_1}$, $(c_i)_{i \in I_1}$, $(e_i)_{i \in I_2}$, $(f_i)_{i \in I_2}$, $(g_i)_{i \in I_2}$, $(u_i)_{i \in I_3}$, $(v_i)_{i \in I_3}$ belong to $Z^1_B(\pi_* \Omega^{\bullet}_\ms X)$.
\end{Lemma}
\begin{proof}
	In order to show that a pair $(\omega, f) \in \Omega_{\ms X}(V) \times k(\ms X)$ is in 
	$Z^1_B(\pi_* \Omega^{\bullet}_\ms X)$, one has to show that 
	$\omega - df \in \Omega_{\ms X, P}$. 
	We have:
	\begin{itemize}[leftmargin=*]
		\item $a_i \in Z^1_B(\pi_* \Omega^{\bullet}_\ms X)$ for $i \in I_1$, since $\ord_P(x^i \, dx) = m + 2M - 5 - 4i \ge 0$,
		
		\item $b_i \in Z^1_B(\pi_* \Omega^{\bullet}_\ms X)$ for $i \in I_1$, since:
		\begin{align*}
			y_1 x^i \, dx - d \left( \frac{y_0 y_1 x^i \alpha}{h_1' + \alpha h_0'} \right)
			= \frac{h_1' x^i (y_1 + \alpha y_0) \, dx}{h_1' + \alpha h_0'} + y_0 y_1 \cdot d\left( \frac{x^i \alpha}{h_1' + \alpha h_0'} \right)  \in \Omega_{\ms X, P}
		\end{align*}
		(the first summand has valuation $2m - 4i - 5 \ge 0$, and the second $2M + m - 4i - 1 \ge 0$),
		
		\item $c_i \in Z^1_B(\pi_* \Omega^{\bullet}_\ms X)$ for $i \in I_1$, since:
		\begin{align*}
			y_0 x^i \, dx - d \left( \frac{y_0 y_1 x^i}{h_1' + \alpha h_0'} \right)
			= \frac{x^i h_0' (y_1 + \alpha y_0) \, dx}{h_1' + \alpha h_0'}
			+ y_0 y_1 \cdot d \left( \frac{x^i}{h_1' + \alpha h_0'} \right) \in \Omega_{\ms X, P}
		\end{align*}
		(the first summand has valuation $4M - 2m -4i - 5 {\ge} 0$ and the second $4M - m - 4i - 1 {\ge} 0$),
		\item $e_i \in Z^1_B(\pi_* \Omega^{\bullet}_\ms X)$ for $i \in I_2$
		and $u_i \in Z^1_B(\pi_* \Omega^{\bullet}_\ms X)$ for $i \in I_3$, since:
		\begin{align*}
			\left[\frac{h_1'}{x^i} \right]_{\ge 0} \, dx - d \left(\frac{y_1}{x^i} \right)
			= \left[\frac{h_1'}{x^i} \right]_{< 0} \, dx
			+ i \cdot \frac{y_1 \, dx}{x^{i+1}}  \in \Omega_{\ms X, P}
		\end{align*}
		(the first summand is holomorphic at $P$ by~\eqref{eqn:P<0_is_holo}, the second has valuation $m + 4i - 1 {\ge} 0$),
		\item $f_i \in Z^1_B(\pi_* \Omega^{\bullet}_\ms X)$ for $i \in I_2$, since:
		\begin{align*}
			\left[\frac{h_0'}{x^i} \right]_{\ge 0} \, dx -
			d \left(\frac{y_0}{x^i} \right) &=
			\left[\frac{h_0'}{x^i} \right]_{< 0} \, dx +
			i \cdot \frac{y_0 \, dx}{x^{i+1}}  \in \Omega_{\ms X, P},
		\end{align*}
		(the first summand is holomorphic at $P$ by~\eqref{eqn:P<0_is_holo}, the second has valuation
		$2M - m + 4i - 1 \ge 0$),		
		\item $g_i \in Z^1_B(\pi_* \Omega^{\bullet}_\ms X)$ for $i \in I_2$, since:
		\begin{align*}
			&y_1 \cdot \left[\frac{h_0'}{x^i} \right]_{\ge 0} \, dx + y_0 \left[\frac{h_1'}{x^i} \right]_{\ge 0} \, dx
			- d \left( \frac{y_0 y_1}{x^i} + \frac{i \cdot y_1 \alpha h_0}{h_1' \cdot x^{i+1}} \right)\\
			&= y_1 \cdot \left[\frac{h_0'}{x^i} \right]_{< 0} \, dx + y_0 \left[\frac{h_1'}{x^i} \right]_{< 0} \, dx
			+ i \cdot \frac{y_0 y_1 + \alpha h_0}{x^{i+1}} \, dx + i \cdot y_1 \cdot d \left(\frac{\alpha h_0}{h_1' \cdot x^{i+1}}\right) \in \Omega_{\ms X, P}
		\end{align*}
		(the first two terms are holomorphic by~\eqref{eqn:P<0_is_holo}, the third term has valuation $4i - 1$
		and the fourth term valuation $2M - m + 4i + 3 \ge 0$),
		\item $v_i \in Z^1_B(\pi_* \Omega^{\bullet}_\ms X)$ for $i \in I_3$, since:
		\begin{align*}
			&y_0 \left[ \frac{h_1'}{x^i} \right]_{\ge 0} \, dx
			 - d \left(\frac{y_0 y_1}{h_1' + \alpha h_0'} \left[ \frac{h_1'}{x^i} \right]_{\ge 0} \right)\\
			&= \left[ \frac{h_1'}{x^i} \right]_{\ge 0} \cdot \frac{h_0' \cdot (y_1 + \alpha y_0)}{h_1' + \alpha h_0'} \, dx
			+ y_0 y_1 \cdot d \left( \frac{\left[ h_1'/x^i \right]_{\ge 0}}{h_1' + \alpha h_0'} \right) \in \Omega_{\ms X, P}
		\end{align*}
		($\left[ \frac{h_1'}{x^i} \right]_{\ge 0}$ is a polynomial of degree $M-1-i$. Hence the first
		summand has valuation $4i - 2m - 1 \ge 0$ and the second one valuation $4i - m + 3 \ge 0$).
	\end{itemize}
\end{proof}
\begin{Lemma} \label{lem:automorphisms_on_abc_efg_uv}
	We have the following equalities in $H^1_{dR}(\ms X)$:
	\begin{alignat*}{3}
		\sigma(a_i) = a_i, \quad &\sigma(b_i) = b_i, \quad &\sigma(c_i) = c_i + a_i, \\
		\tau(a_i) = a_i, \quad &\tau(b_i) = b_i + a_i, \quad &\tau(c_i) = c_i, \\
		\sigma(e_i) = e_i, \quad &\sigma(f_i) = f_i, \quad &\sigma(g_i) = g_i + e_i, \\
		\tau(e_i) = e_i, \quad &\tau(f_i) = f_i, \quad &\tau(g_i) = g_i + f_i, \\
		\sigma(u_i) = u_i, \quad &\sigma(v_i) = v_i + u_i,\\
		\tau(u_i) = u_i, \quad &\tau(v_i) = v_i.
	\end{alignat*}
\end{Lemma}
\begin{proof}
	It is straightforward that $a_i$ is $\VV_4$-invariant for $i \in I_1$. Note now that
	for $i \in I_1$:
	\begin{align*}
		\sigma(b_i) - b_i = \left(0, \frac{y_1 x^i \alpha}{h_1' + \alpha h_0'} \right).
	\end{align*}
	Moreover:
	\begin{align*}
		\ord_P\left(\frac{y_1 x^i \alpha}{h_1' + \alpha h_0'}\right) &= 2m - 4i - 4 \ge 0.
	\end{align*}
	Hence $\sigma(b_i) = b_i$ in $H^1_{dR}(\ms X)$. Similarly:
	\begin{itemize}[leftmargin=*]
		\item $\tau(b_i) = b_i + a_i$ for $i \in I_1$, since:
		\begin{align*}
			\ord_P\left(\frac{y_0 x^i \alpha}{h_1' + \alpha h_0'}\right) &= 2M - 4i - 4 \ge 0.
		\end{align*}
		\item $\sigma(c_i) = c_i + a_i$, $\tau(c_i) = c_i$ for $i \in I_1$, since:
		\begin{align*}
			\ord_P\left(\frac{y_1 x^i}{h_1' + \alpha h_0'}\right) &= 2M - 4i - 4 \ge 0,\\
			\ord_P\left(\frac{y_0 x^i}{h_1' + \alpha h_0'}\right) &= 4M - 2m - 4i - 4 \ge 0,
		\end{align*}
		\item $\sigma(e_i) = \tau(e_i) = e_i$, $\sigma(f_i) = \tau(f_i) = f_i$ for $i \in I_2$
		and $\sigma(u_i) = \tau(u_i) = u_i$ for $i \in I_3$, since $1/x^i \in \mc O_{\ms X, P}$,
				
		\item $\sigma(g_i) = g_i + e_i$, $\tau(g_i) = g_i + f_i$ for $i \in I_2$, since:
		\[
		\ord_P\left(\frac{\alpha h_0}{h_1' \cdot x^{i+1}}\right) = 2 (M-m) + 4i \ge 0,
		\]
		\item $\sigma(v_i) = v_i + u_i$, $\tau(v_i) = v_i$ for $i \in I_3$, since:
		\begin{align*}
			\ord_P\left(\frac{y_1}{h_1' + \alpha h_0'}
			\left[ \frac{h_1'}{x^i} \right]_{\ge 0} - \frac{y_1}{x^i}\right) &\ge
			0,\\
			\ord_P\left(\frac{y_0}{h_1' + \alpha h_0'}
			\left[ \frac{h_1'}{x^i} \right]_{\ge 0}\right) &= 4M - 2m - 4i - 4 \ge 0.
		\end{align*}
		For the first inequality, note that:
		\begin{align*}
			\frac{y_1}{h_1' + \alpha h_0'}
			\left[ \frac{h_1'}{x^i} \right]_{\ge 0} - \frac{y_1}{x^i}
			= \frac{y_1 \left[ h_1'/x^i \right]_{<0}}{h_1' + \alpha h_0'} +  \frac{y_1 \alpha h_0'}{(h_1' + \alpha h_0') \cdot x^i}.
		\end{align*}
		The first summand has valuation at least $2M$
		at $P$, the second $4i - 2m \ge 0$. \qedhere
	\end{itemize}
\end{proof}
In the sequel we use the following basis of the holomorphic differentials on $\ms X$.
\begin{Theorem} \label{thm:basis_H0_Omega_HKG}
	The basis of $H^0(\ms X, \Omega_\ms X)$ is given by the differential forms:
	\begin{itemize}
		\item $x^i \, dx$ for $i = 0, \ldots, M - M_0 - 2$,
		\item $y_0 x^i \, dx$ for $i = 0, \ldots, M_0 - 1$,
		\item $(y_1 + \alpha y_0) \cdot x^i \, dx$ for $i = 0, \ldots, m' - 1$.
	\end{itemize}
\end{Theorem}
\begin{proof}
	This is basically \cite[Lemma~3.6]{Bleher_Camacho_Holomorphic_differentials}. The cited article uses the following notation:
	$u = u_{\infty} := y_0$, $v = v_{\infty} := y_1$, $p := h_0$, $q := h_1$, $\pi_{\infty} := \frac{1}{x}$, $w_{\infty}(j) := y_1 + [\alpha]_{\ge M'' - j} \cdot y_0$ for $j \in \ZZ$, where $[\alpha]_{\ge t} := \sum_{i = t}^{M''} \alpha_i x^i$ for any $t \in \ZZ$. Thus by \cite[Lemma~3.6]{Bleher_Camacho_Holomorphic_differentials} the basis of $\ms X$ is given by $\{ f \, dx : f \in \mc B \}$,
	where:
	\begin{align*}
		\mc B = \mc B_{\infty} &= \mc B_{\infty, 1} \cup \mc B_{\infty, 2}
		\cup \mc B_{\infty, 3},\\
		\mc B_{\infty, 1} &= 
		\left\{ x^{i_1} : 0 \le i_1 \le M - M_0 - 2 \right\},\\
		\mc B_{\infty, 2} &= \left\{ y_0 \cdot x^{i_2} : 0 \le i_2 \le M_0 - 1 \right\}, \\
		\mc B_{\infty, 3} &= \left\{ y_1 \cdot x^{i_3} : 0 \le i_3 \le \left\lfloor \frac{m - 5}{4} \right\rfloor  \right\}\\
		&\cup \left\{ w_{\infty}\left(i_3 - \left\lfloor \frac{m - 1}{4} \right\rfloor \right) \cdot x^{i_3} : \left\lfloor \frac{m - 1}{4} \right\rfloor
		\le i_3 \le m' - 1 \right\}.
	\end{align*}
	By adding the elements of $\mc B_{\infty, 2}$ to the elements of $\mc B_{\infty, 3}$, one may replace $\mc B_{\infty, 3}$ by
	\[
	\{ (y_1 + \alpha y_0) \cdot x^i : 0 \le i \le m' - 1 \}.
	\]
	This ends the proof.
\end{proof}
\noindent Observe that by~\eqref{eqn:residue_and_trace}:
\begin{align}
	\res_P\left(\frac{y_0 y_1}{x^i} \, dx\right) =
	\begin{cases}
		1, & i = 1,\\
		0, & \textrm{otherwise,}	
	\end{cases} \label{eqn:residues_z0_z1}\\
	\res_P\left(\frac{y_0^2}{x^i}  \, dx\right) = \res_P\left(\frac{y_1^2}{x^i}  \, dx\right) = \res_P\left(\frac{dx}{x^i}\right) = 0
	\label{eqn:residues_z0_z1_2}
\end{align}
for any $i \in \ZZ$. We recall also that the Serre duality is given by the pairing:
\begin{align*}
	H^1(\ms X, \mc O_{\ms X}) \times H^0(\ms X, \Omega_\ms X) \to k,\\
	\langle f, \omega \rangle := \res_P(f \cdot \omega).
\end{align*}
In order to prove that the elements listed above yield a basis of $H^1_{dR}(\ms X)$,
we start by establishing the linear independence of some of them. The idea is to use the
Hodge--de Rham exact sequence:
\[
	0 \to H^0(\ms X, \Omega_{\ms X}) \to H^1_{dR}(\ms X) \to H^1(\ms X, \mc O_{\ms X}) \to 0.
\]
First, we investigate the dependence of the images of the elements in $H^1(\ms X, \mc O_{\ms X})$
and use Serre duality along with Theorem~\ref{thm:basis_H0_Omega_HKG}. Then we are left with a dependence in
$H^0(\ms X, \Omega_{\ms X})$ and we may use Theorem~\ref{thm:basis_H0_Omega_HKG}
for the second time.
\begin{Lemma} \label{lem:afeu_linearly_independent}
	The elements $(a_i)_{i \in I_1}$, $(e_i)_{i \in I_2}$, $(f_i)_{i \in I_2}$, $(u_i)_{i \in I_3}$ are linearly independent in $H^1_{dR}(\ms X)$.
\end{Lemma}
\begin{proof}
	We extend the definition of $a_i$ and $e_i$ as follows.
	Denote $a_i = (x^i \, dx, 0)$ for $i = m', \ldots, M' - 1$ and
	 $e_i := u_i$ for $i = m'+1, \ldots, M'$. Note that for $i > M'$:
	\[
		u_i = \left(\left[ \frac{h_1'}{x^i} \right]_{\ge 0} \, dx, \frac{y_1}{x^i} \right) = \left(\left[ \frac{h_1'}{x^i} \right]_{\ge 0} \, dx, 0 \right)
		\quad \textrm{ in } H^1_{dR}(\ms X).
	\]
	Moreover, $\left[ \frac{h_1'}{x^i} \right]_{\ge 0}$ is a polynomial of degree $M - 1 - i$.
	Hence:
	\[
		u_i \in \Span_k(a_j : j \le M-1-i) \setminus \Span_k(a_j : j < M-1-i).
	\]
	Therefore it suffices to show that the elements $(a_i)_{i = 0}^{M' - 1}$, $(e_i)_{i = 1}^{M'}$, $(f_i)_{i = 1}^{m'}$ are linearly independent
	in $H^1_{dR}(\ms X)$. Suppose that for some $A_i, F_i, E_i \in k$ we have:
	\begin{equation} \label{eqn:dependence_in_HdR}
		\sum_{i = 0}^{M' - 1} A_i \cdot a_i + \sum_{i = 1}^{M'} E_i \cdot e_i + \sum_{i = 1}^{m'} F_i \cdot f_i = 0
	\end{equation}
	in $H^1_{dR}(\ms X)$. Then in $H^1(\ms X, \mc O_{\ms X})$:
	\begin{equation} \label{eqn:dependence_in_H1OX}
		\sum_{i = 1}^{M'} E_i \cdot \frac{y_1}{x^i} + \sum_{i = 1}^{m'} F_i \cdot \frac{y_0}{x^i} = 0.
	\end{equation}
	We divide the proof into four steps.\\ \mbox{} \\
	\bb{Step I:} For $j = 1, \ldots, M_0$ we have $E_j = 0$
	and for every $j = 1, \ldots, m'$:
	\begin{equation} \label{eqn:gamma=alpha_rho}
		F_j = \sum_{l = 0}^{M''} \alpha_l \cdot E_{j + l}.
	\end{equation}
	\bb{Proof of Step I:} Using Theorem~\ref{thm:basis_H0_Omega_HKG}, \eqref{eqn:dependence_in_H1OX} and relations~\eqref{eqn:residues_z0_z1}, \eqref{eqn:residues_z0_z1_2} we obtain for any $j \in \{ 1, \ldots, M_0 \}$:
	\begin{align*}
		0 = \left\langle \sum_{i = 1}^{M'} E_i \cdot \frac{y_1}{x^i} + \sum_{i = 1}^{m'} F_i \cdot \frac{y_0}{x^i}, \quad y_0 \cdot x^{j - 1} \, dx \right\rangle = E_j.
	\end{align*}
	Analogously, for any $j = 1, \ldots, m'$:
	\begin{align*}
		0 &= \left\langle \sum_{i = 1}^{M'} E_i \cdot \frac{y_1}{x^i} + \sum_{i = 1}^{m'} F_i \cdot \frac{y_0}{x^i}, \quad
		(y_1 + \alpha y_0) \cdot x^{j - 1} \, dx \right\rangle\\
		&= \sum_{l = 0}^{M''} \alpha_l \cdot E_{j + l} + F_j.
	\end{align*}
	This shows that~\eqref{eqn:gamma=alpha_rho} is true.\\ \mbox{} \\
	\bb{Step II:} For any $i = 1, \ldots, M'$:
	\begin{equation} \label{eqn:combination_in_OXP}
		E_i \cdot \sum_{l = i_0(i)}^{i - 1} \alpha_l \cdot \frac{y_0}{x^{i - l}}
		\equiv E_i \cdot \frac{y_1}{x^i} + y_0 \cdot P_i(x) \pmod{\mc O_{\ms X, P}},
	\end{equation}
	where $i_0(i) := \max(0, i - m')$, $P_i \in k[x]$, $\deg P_i \le M'' - i$.\\
	\bb{Proof of Step II:} Take $P_i(x) := E_i \cdot \sum_{l = i}^{M''} \alpha_l \cdot x^{l - i}$. 
	For $i \le M_0$ the equality~\eqref{eqn:combination_in_OXP} is immediate, since $E_i = 0$
	by Step~I. If $i > M_0$, then:
	\begin{align*}
		E_i \cdot \sum_{l = i_0(i)}^{i - 1} \alpha_l \cdot \frac{y_0}{x^{i - l}} + E_i \cdot \frac{y_1}{x^i} + y_0 \cdot P_i(x)
		= E_i \cdot \left(\frac{y_1 + \alpha y_0}{x^i} + \sum_{l = 0}^{i_0(i) - 1} \alpha_l \cdot \frac{y_0}{x^{i - l}} \right).
	\end{align*}
	But $\ord_P(\frac{y_1 + \alpha y_0}{x^i}) \ge 4 (M_0 + 1) -2M + m \ge 0$. Moreover, for $0 \le l \le i_0(i) - 1$ one has $i - l > m'$, which implies that $y_0/x^{i - l} \in \mc O_{\ms X, P}$.
	This proves~\eqref{eqn:combination_in_OXP}.\\  \mbox{} \\
	\bb{Step III:} There exist polynomials $S_i, T_i \in k[x]$,
	$\deg T_i \le M - i$ such that in $H^1_{dR}(\ms X)$:
	\begin{equation*}
		\sum_{i = 1}^{m'} F_i \cdot f_i
		= \sum_{i = 1}^{M'} E_i \cdot \left(y_0 S_i(x) \, dx
		+ T_i(x) \, dx, \frac{y_1}{x^i} \right).
	\end{equation*}
	\bb{Proof of Step III:} Using~\eqref{eqn:gamma=alpha_rho}:
	\begin{align*}
		\sum_{i = 1}^{m'} F_i \cdot f_i
		&= \sum_{i = 1}^{m'} \sum_{l = 0}^{M''} \alpha_l E_{l + i} \cdot \left([h_0'/x^i]_{\ge 0} \, dx, \frac{y_0}{x^i}\right).
	\end{align*}
	Thus substituting $t := l+i$:
	\begin{align*}
		\sum_{i = 1}^{m'} F_i \cdot f_i
		&= \sum_{l = 0}^{M''} \sum_{t = l + 1}^{l + m'} \alpha_l E_t \cdot \left([h_0'/x^{t-l}]_{\ge 0} \, dx, y_0/x^{t - l} \right) \\
		&= \sum_{t = 1}^{m' + M''} \sum_{l = i_0(t)}^{t - 1} \alpha_l E_t \cdot \left([h_0'/x^{t-l}]_{\ge 0} \, dx, \frac{y_0}{x^{t - l}} \right)\\
		&= \sum_{t = 1}^{M'} E_t \cdot \left(Q_t(x) \, dx, \sum_{l = i_0(t)}^{t - 1} \alpha_l \frac{y_0}{x^{t - l}} \right),
	\end{align*}
	where $Q_t(x) := \sum_{l = i_0(t)}^{t - 1} \alpha_l [h_0'/x^{t-l}]_{\ge 0} \in k[x]$. Moreover,
	since $\alpha_i = 0$ for $i > M''$, we have $\deg Q_t \le \deg [h_0'/x^{t-M''}]_{\ge 0} = m + M'' - 1 - t$.
	By~\eqref{eqn:combination_in_OXP}:
	\begin{align*}
		\sum_{i = 1}^{m'} F_i \cdot f_i
		&= \sum_{i = 1}^{M'} E_i \cdot \left(Q_i(x) \, dx, \frac{y_1}{x^i} + y_0 \cdot P_i(x) \right)\\
		&= \sum_{i = 1}^{M'} E_i \cdot \left(Q_i(x) \, dx + d(y_0 \cdot P_i(x)), \frac{y_1}{x^i} \right)\\
		&= \sum_{i = 1}^{M'} E_i \cdot \left(y_0 S_i(x) \, dx
		+ T_i(x) \, dx, \frac{y_1}{x^i} \right),
	\end{align*}
	where $S_i(x) := P_i'(x)$, $T_i(x) := Q_i(x) + h_0'(x) \cdot P_i(x)$, $\deg T_i \le m + M'' -1 - i < M - i$.\\ \mbox{} \\
	\bb{Step IV:} By applying the equality from Step III
	to~\eqref{eqn:dependence_in_HdR} we obtain:
	\begin{align*}
		0&= \sum_{i = 0}^{M' - 1} A_i \cdot (x^i \, dx, 0) + \sum_{i = 1}^{M'} E_i \cdot ([h_1'/x^i]_{\ge 0} \, dx + y_0 S_i(x) \, dx
		+ T_i(x) \, dx, 0).
	\end{align*}
	In other words:
	\[
		\sum_{i = 0}^{M' - 1} A_i x^i \, dx + \sum_{i = 1}^{M'} E_i \cdot ([h_1'/x^i]_{\ge 0} \, dx + y_0 S_i(x) \, dx
		+ T_i(x) \, dx) = 0
		\qquad \textrm{ in } H^0(\ms X, \Omega_\ms X).
	\]
	Note that the leading term of $[h_1'/x^i]_{\ge 0} \, dx$ is $x^{M - 1 - i} \, dx$.
	But by Theorem~\ref{thm:basis_H0_Omega_HKG}, for any $i \in \{ 1, \ldots, M' \}$, the form
	$x^{M - 1 - i} \, dx$ is not in the $k$-linear span of the forms $x^j \, dx$ for $j < M - 1 - i$
	and $y_0 x^l \, dx$ for $l = 0 , \ldots, M_0 - 1$. Thus $E_i = 0$ for any $i$. Analogously,
	Theorem~\ref{thm:basis_H0_Omega_HKG} implies that $A_0 = \ldots = A_{M' - 1} = 0$.
\end{proof}

\begin{Corollary} \label{cor:basis_of_HKG_dR}
	The images of the elements $(a_i)_{i \in I_1}$, $(b_i)_{i \in I_1}$, $(c_i)_{i \in I_1}$, $(e_i)_{i \in I_2}$, $(f_i)_{i \in I_2}$, $(g_i)_{i \in I_2}$, $(u_i)_{i \in I_3}$, $(v_i)_{i \in I_3}$
	form a basis of $H^1_{dR}(\ms X)$.
\end{Corollary}
\begin{proof}
	Note that the number of listed elements is
	\begin{align*}
		3 \cdot |I_1| + 3 \cdot |I_2| + 2 \cdot |I_3| &=
		3 \cdot m' + 3 \cdot m' + 2 \cdot (M - m) = 2 g_\ms X.
	\end{align*}
	Thus it suffices to show that the listed elements are linearly independent.
	Suppose to the contrary that for some $A_i, B_i, C_i, E_i, F_i, G_i, U_i, V_i \in k$:
	\begin{align*}
		0 = \sum_{i \in I_1} (A_i \cdot a_i + B_i \cdot b_i + C_i \cdot c_i) + \sum_{i \in I_2} (E_i \cdot e_i + F_i \cdot f_i + G_i \cdot g_i) + \sum_{i \in I_3} (U_i \cdot u_i + V_i \cdot v_i).
	\end{align*}
	Then, by applying $(\sigma - \id)$ we obtain by Lemma~\ref{lem:automorphisms_on_abc_efg_uv}:
	\[
		0 = \sum_{i \in I_1} C_i \cdot a_i + \sum_{i \in I_2} G_i \cdot e_i 
		+ \sum_{i \in I_3} V_i \cdot u_i.
	\]
	Therefore by Lemma~\ref{lem:afeu_linearly_independent} $C_i = G_i = V_i = 0$. Analogously, by applying $(\tau - \id)$ we obtain
	$B_i = 0$. Hence we obtain:
	\begin{align*}
		0 = \sum_{i \in I_1} A_i \cdot a_i + \sum_{i \in I_2} (E_i \cdot e_i + F_i \cdot f_i) + \sum_{i \in I_3} U_i \cdot u_i.
	\end{align*}
	The proof follows by applying Lemma~\ref{lem:afeu_linearly_independent} one more time.
\end{proof}
\begin{proof}[Proof of {Proposition~\ref{prop:dR_of_Klein_HKG}}]
	Lemma~\ref{lem:automorphisms_on_abc_efg_uv} shows that the maps:
	\begin{align*}
		N_{2, 0} \to H^1_{dR}(\ms X), \qquad& (1, 0) \mapsto u_i, 	&(0, 1) \mapsto v_i, 	& & \textrm{ for any } i \in I_1,\\
		M_{3, 1} \to H^1_{dR}(\ms X), \qquad& (1, 0, 0) \mapsto a_i, 	&(0, 1, 0) \mapsto b_i, \quad &(0, 0, 1) \mapsto c_i, & \textrm{ for any } i \in I_2,\\
		M_{3, 2} \to H^1_{dR}(\ms X), \qquad& (1, 0, 0) \mapsto e_i, 	&(0, 1, 0) \mapsto f_i, \quad &(0, 0, 1) \mapsto g_i. & \textrm{ for any } i \in I_3.
	\end{align*}
	are equivariant. The direct sum of those maps is a $k[\VV_4]$-linear map:
	\[
		N_{2, 0}^{M - m} \oplus M_{3, 1}^{m'} \oplus M_{3, 2}^{m'}
		\to H^1_{dR}(\ms X),
	\]
	which is an isomorphism by Corollary~\ref{cor:basis_of_HKG_dR}.
\end{proof}
\bibliography{bibliografia}
\end{document}